\documentclass{jgcc}

\pdfoutput=1

\usepackage{lastpage}

\jgccdoi{14}{2}{2}{10019}

\jgccheading{}{\pageref{LastPage}}{}{}{Sept.~8,~2022}{Feb.~1,~2023}{}  


\usepackage{amsfonts,amsmath,amssymb,amsthm}
\usepackage{mathtools}
\usepackage{derivative}
\allowdisplaybreaks

\usepackage{enumitem}
\usepackage{tikz}
\usepackage{tikz-cd}
\usetikzlibrary{automata, positioning, arrows, calc}

\newcommand{\N}{\mathbb{N}}
\newcommand{\Z}{\mathbb{Z}}
\DeclarePairedDelimiter{\abs}{|}{|}

\DeclareMathOperator{\Lk}{Lk}
\DeclareMathOperator{\supp}{supp}

\DeclareMathOperator{\NGP}{NGP}
\DeclareMathOperator{\RACG}{RACG}

\usepackage[margin=1.25in]{geometry}
\usepackage{mathrsfs}
\renewcommand{\mathcal}{\mathscr}
\usepackage[justification=centering]{caption}

\usepackage{hyperref}

\title{Geodesic Growth of Numbered Graph Products}

\begin{document}
\author[L. Marjanski]{Lindsay Marjanski}
\address{College of the Holy Cross, Worcester, MA 01610, USA}

\author[V. Solon]{Vincent Solon}
\address{The University of Texas at Austin, Austin, TX 78712, USA}

\author[F. Zheng]{Frank Zheng}
\address{University of California, Los Angeles, Los Angeles, CA 90095, USA}

\author[K. Zopff]{Kathleen Zopff}
\address{Bellarmine University, Louisville, KY 40205, USA}
 
\begin{abstract}
    In this paper, we study geodesic growth of numbered graph products; these are a generalization of right-angled Coxeter groups, defined as graph products of finite cyclic groups. We first define a graph-theoretic condition called link-regularity, as well as a natural equivalence amongst link-regular numbered graphs, and show that numbered graph products associated to equivalent link-regular numbered graphs must have the same geodesic growth series. Next, we derive a formula for the geodesic growth of right-angled Coxeter groups associated to equivalent link-regular graphs. Finally, we find a system of equations that can be used to solve for the geodesic growth of numbered graph products corresponding to link-regular numbered graphs that contain no triangles and have constant vertex numbering.
\end{abstract}

\maketitle

\tableofcontents

\section{Introduction}
Given a group $G$ and a generating set $S$, one can define the \textit{standard growth series} of $(G,S)$ to be the generating function whose $n$th coefficient counts the elements of $G$ with word length $n$. The standard growth series of $G$ reflects a number of important properties of $G$ itself; a celebrated result in this area is Gromov's polynomial growth theorem which states that groups with growth series of polynomial order are virtually nilpotent \cite{Gro81}. Similarly, one can study the \textit{geodesic growth series} of $(G,S)$, which is the generating function whose $n$th coefficient counts the number of length $n$ \textit{geodesics}, or geodesic paths of length $n$ in the Cayley graph of $(G, S)$ starting at the identity. 

A \textit{numbered graph product} (NGP) is a certain group defined from a \textit{numbered graph}. A numbered graph is a pair $(\Gamma, N)$ where $\Gamma$ is a simplicial graph and $N: V\Gamma \rightarrow \N$ is a map that assigns a natural number greater than 1 to each vertex of $\Gamma$. The NGP associated to $(\Gamma, N)$, denoted by $\NGP(\Gamma, N)$ or just $\NGP(\Gamma)$, is the group with presentation
\[
    \left\langle v \in V\Gamma \bigm| \text{$v^{N(v)} = 1$ for all $v \in V\Gamma$, $uv = vu$ for all $\{u, v\} \in E\Gamma$}\right\rangle.
\]
We call the set $S = V\Gamma \cup (V\Gamma)^{-1}$ the \textit{standard generating set} of $\NGP(\Gamma)$. A \textit{right-angled Coxeter group} (RACG) is an NGP where $N(v) = 2$ for all $v \in V\Gamma$.

Throughout this paper, we will be interested in the geodesic growth series of NGPs associated to numbered graphs satisfying a strong combinatorial property known as \textit{link-regularity}. Our first result builds on the work of \cite{Ant12}. There, the authors show that the geodesic growth of a RACG associated to a graph $\Gamma$ is fully determined by combinatorial properties of $\Gamma$ related to link-regularity. This result was then used to find distinct RACGs with the same geodesic growth.

In Section \ref{sec: theorem 1}, we extend this result to NGPs, when the associated numbered graph satisfies a generalized definition of link-regularity:

\setcounter{section}{5}
\setcounter{thm}{2}
\begin{defi}
    Let $(\Gamma, N)$ be a numbered graph. If $U \subseteq V\Gamma$, we let $N(U)$ denote the multiset $\{N(v) \mid v \in U\}$. We say $\Gamma$ is \textit{link-regular} if, whenever $N(\sigma_1) = N(\sigma_2)$ for cliques $\sigma_1,\sigma_2$, we have $N(\Lk(\sigma_1)) = N(\Lk(\sigma_2))$. (Here, $\Lk(\sigma)$ denotes the \textit{link} of the clique $\sigma$, defined as $\Lk(\sigma) = \{ v \in V\Gamma \setminus \sigma \mid \{v\} \cup \sigma \text{ is a clique}\}$, cf. Definition \ref{def: cliques and links}.)
\end{defi}

This definition is equivalent to the definition of link-regularity in \cite{Ant12} when the NGP in question is a RACG. We use this condition to define a certain equivalence (see Definition \ref{def: graph equivalence}) on link-regular numbered graphs. The main result of Section \ref{sec: theorem 1} is the following:

\setcounter{thm}{13}
\begin{thm}
    If $(\Gamma, N)$ and $(\Gamma', N')$ are equivalent link-regular numbered graphs, then $\NGP(\Gamma)$ and $\NGP(\Gamma')$ have the same geodesic growth series.
\end{thm}

The equivalence of link-regular numbered graphs is weaker than the isomorphism, allowing us to find infinitely many distinct graphs with identical geodesic growth series. We prove Theorem \ref{theorem: main} by sequentially generalizing the arguments from \cite{Ant12} to allow for the use of multisets instead of a single number.

Our next result builds on Section 5 of \cite{Ant12}, where the authors find an explicit formula for the geodesic growth series of a RACG associated to a link-regular, triangle-free graph $\Gamma$. The paper \cite{Ant13} improves upon this result by allowing $\Gamma$ to be tetrahedron-free. In Section \ref{sec: theorem 2}, we further strengthen this result by obtaining the geodesic growth series of a RACG based on any link-regular graph $\Gamma$. In particular, we have the following theorem:

\setcounter{section}{6}
\setcounter{thm}{3}
\begin{thm}
    Let $\Gamma$ be a link-regular simplicial graph with maximum clique size $d$. Let $\ell_0 = |V(\Gamma)|$, and for $1 \leq k \leq d$, let $\ell_k = \abs{\Lk(\sigma)}$, where $\sigma$ is any $k$-clique (this is well-defined by link-regularity of $\Gamma$). For integers $0 \leq m \leq d$ and $0 \leq k \leq m$, set
    \[
        N_{m,k} =
        \begin{dcases}
            \left(\prod_{k < j < m} \ell_j\right)\sum_{j = k}^m \binom{m - k}{j - k} (-1)^{j - k} \ell_j
                        & \text{if $k < m - 1$} \\
            \ell_{m - 1} - \ell_m - 1 & \text{if $k = m - 1$} \\
            1 & \text{if $k = m$},
        \end{dcases}
    \]
    and for integers $0 \leq i \leq d$ and $0 \leq j \leq i$, set
    \[
        \begin{dcases}
            M_{i, j} = \sum_{\substack{j \leq s_1 < t_1 \leq \dots \leq s_n < t_n \leq d \\
                (t_1 - s_1) + \dots + (t_n - s_n) = i - j}}
        (-1)^n \binom{t_1}{s_1} \cdots \binom{t_n}{s_n} N_{t_1, s_1} \cdots N_{t_n, s_n}
        & \text{if $i \neq j$} \\
            M_{i, j} = 1 & \text{if $i = j$.} \\
        \end{dcases}
    \]
    Then the geodesic growth of $\RACG(\Gamma)$ is given by the rational function
    \[
        \mathcal{G}(z) =
        \frac{\sum_{i = 0}^d \left(\sum_{j = 0}^i \ell_0 \cdots \ell_{j - 1} M_{i, j}\right)z^i}
        {\sum_{i = 0}^d M_{i, 0} z^i}.
    \]
\end{thm}

To prove Theorem \ref{thm: main RACG theorem}, we apply Theorem \ref{theorem: chomsky black magic}, a method due to Chomsky and Schützenberger to the grammar that generates the geodesic words of $G$. This yields a large system of equations that can be solved to obtain the geodesic growth series of $G$. Using the fact that $\Gamma$ is link-regular, we apply combinatorial arguments to obtain a simpler system of $d + 1$ equations, where $d$ is the maximal clique size of $\Gamma$. Finally, we rearrange this system into a polynomial recurrence relation, which we solve via induction.

Lastly, in Section \ref{sec: theorem 3}, we focus on NGPs when $(\Gamma, N)$ is a triangle-free, link-regular numbered graph such that the vertex numbering $N$ is constant. We obtain the following system of equations, which can be used to solve for the geodesic growth of $\NGP(\Gamma, N)$:

\setcounter{section}{7}
\setcounter{thm}{0}
\begin{thm}
    Let $\Gamma$ be an $L$-regular (i.e., every vertex in $\Gamma$ has $L$ neighbors) and triangle-free simplicial graph with $n$ vertices. Given an integer $N \geq 2$, let $(\Gamma, N)$ be the numbered graph where all vertex numbers are $N$. The geodesic growth series $\mathcal{G}(z)$ of $\NGP(\Gamma)$ can be determined by solving the following system of equations: for $1 \leq k, \ell \leq \lfloor N / 2 \rfloor$, we have
    \begin{align*}
        \mathcal{G}(z) &= 2z\mathcal{G}_1(z) + 1 \\
        \mathcal{G}_k(z) &= z(\mathcal{G}_{k + 1}(z)
        + 2\mathcal{G}_{1, k}(z) + 2(n - L - 1)\mathcal{G}_1(z)) + n \\
        \mathcal{G}_{\{k, \ell\}}(z)
        &= z(\mathcal{G}_{\{k + 1, \ell\}}(z)
        + 2(L - 1)(\mathcal{G}_{\{1, k\}}(z)
        + \mathcal{G}_{\{1, \ell\}}(z))
        + 2 L(n - 2L)\mathcal{G}_1(z))
        + nL.
    \end{align*}
    Here, $\mathcal{G}_k(z)$ and $\mathcal{G}_{\{k, \ell\}}(z)$ are power series such that $\mathcal{G}_{\lfloor N / 2 \rfloor + 1}(z) = \mathcal{G}_{\{\lfloor N / 2 \rfloor + 1, \ell\}}(z) = 0$ for all $1 \leq \ell \leq \lfloor N / 2\rfloor$.
\end{thm}

To derive the system, we once again apply the method in Theorem \ref{theorem: chomsky black magic}. We do not derive a closed-form formula for the geodesic growth, but computer software such as Sage can be used to solve the system for small enough $N$.

We would like to thank and acknowledge Mark Pengitore, Alec Traaseth, and Darien Farnham of the University of Virginia's Mathematics Department for their support and assistance on this paper. Also, we would like to extend our gratitude to all the organizers of the 2022 UVA Topology REU who gave us the opportunity to make this paper possible. Lastly, we acknowledge Yago Antolín for suggesting the problem explored in this paper.

\setcounter{thm}{1}
\setcounter{section}{1}

\section{Notation}
\label{sec: notation}
We compile a list of common notation used throughout this paper:
\begin{itemize}
    \item $\N$ will denote the positive integers (so $0 \not\in \N$).
    \item A group $G$ with a distinguished generating set $S$ is typically denoted by the pair $(G,S)$.
    \item The letter $\Gamma$ usually indicates a simplicial graph (that is, an undirected graph with no double edges or loops). The vertices of $\Gamma$ are denoted $V\Gamma$, and the edges of $\Gamma$ are denoted $E\Gamma$.
    \item A numbered graph is represented by the pair $(\Gamma, N)$. Given a subset $U \subseteq V\Gamma$, $N(U)$ denotes the multiset $\{N(v) \mid v \in U\}$.
    \item The NGP associated to some numbered graph $(\Gamma, N)$ is written $\NGP(\Gamma, N)$, or just $\NGP(\Gamma)$. Similarly, the RACG associated to some graph $\Gamma$ is written $\RACG(\Gamma)$.
    \item The letter $\mathcal{L}$ usually indicates a language, the letter $\mathcal{D}$ usually indicates a finite state automaton, and the letter $C$ usually indicates a grammar.
    \item The letters $\sigma$, $\tau$, and $\pi$ usually indicate cliques. The letter $\Sigma$ usually indicates a powered clique (cf. Definition \ref{def: powered clique}).
    \item The power profile (cf. Definition \ref{def: power profile}) of a powered clique $\Sigma$ is written $P(\Sigma)$. The letters $P$ and $Q$ usually indicate power profiles.
\end{itemize}

\section{Regular languages}
\label{sec: language}
In this section, we develop the basic theory surrounding regular languages. The reader may wish to refer to a standard text such as \cite{Lin22} for a more in-depth discussion on these topics. Our aim will be to define the so-called \textit{growth series} of a language, and to describe a powerful technique (Theorem \ref{theorem: chomsky black magic}) used to compute the growth series of a regular language.

\begin{defi}[Words and languages]
    An \textit{alphabet} is a non-empty finite set, whose elements we call \textit{letters}. A \textit{word} $x$ over an alphabet $\mathcal{A}$ is a finite sequence of the elements from $\mathcal{A}$, usually written in the form $x = a_1 \cdots a_n$. A sequence $a_i a_{i + 1} \cdots a_j$ is called a \textit{subword} of $x$. The \textit{empty word} is denoted by $\varepsilon$. A set of words over $\mathcal{A}$ is a \textit{language} over $\mathcal{A}$. The language consisting of all words over $\mathcal{A}$ is denoted by $\mathcal{A}^*$. For a nonnegative integer $n$, the set of all words of length $n$ in $\mathcal{A}^*$ is denoted by $\mathcal{A}^n$.

    The \textit{length} of a word $x \in \mathcal{A}^*$ is its length when viewed as a sequence, and is written $|x|$. The \textit{concatenation} of two words $x = a_1 \cdots a_n$ and $y = b_1 \cdots b_m$ is the word $xy = a_1 \cdots a_nb_1 \cdots b_m$. Given a word $x$, we write $x^n = x \cdots x$ for word formed by concatenating $x$ to itself $n$ times, with $x^0 = \varepsilon$.
\end{defi}

Regular languages are languages that are simple enough to be described by the following model of computation:

\begin{defi}[Finite state automata]
    A \textit{finite state automaton} (FSA) $\mathcal{D}$ is a $5$-tuple $(Q, \mathcal{A}, \delta, \allowbreak q_0, A)$, where
    \begin{itemize}
        \item $Q$ is a finite set of \textit{states},
        \item $\mathcal{A}$ is an alphabet called the \textit{input alphabet},
        \item $\delta : Q \times \mathcal{A} \to Q$ is the \textit{transition function},
        \item $q_0 \in Q$ is the \textit{start state},
        \item $A \subseteq Q$ is the set of \textit{accept states}.
    \end{itemize}
    If we have $\delta(q, a) = q'$, we will say that $\mathcal{D}$ has an $a$-transition from $q$ to $q'$. Also, we define the \textit{extended transition function} $\delta^* : Q \times \mathcal{A}^* \to Q$ recursively by
    \[
        \delta^*(q, \varepsilon) = q, \quad
        \delta^*(q, xa) = \delta(\delta^*(q, x), a).
    \]
    Given a word $x \in \mathcal{A}^*$, we say that $x$ \textit{ends} at the state $\delta^*(q_0, x)$. If $x$ ends at an accept state, then we say $\mathcal{D}$ \textit{accepts the word $x$}, else, we say $\mathcal{D}$ \textit{rejects the word $x$}. The set of all strings accepted by $\mathcal{D}$ is called the \textit{language accepted by $\mathcal{D}$}.
\end{defi}

\begin{defi}[Regular languages]
    A \textit{regular language} is a language accepted by some FSA.
\end{defi}

Another way of describing a language is through a \textit{grammar}, or a set of rules that describe how words in the language are generated.

\begin{defi}[Regular grammars]
    A \textit{regular grammar} $C$ is a $4$-tuple $(V, \mathcal{A}, \allowbreak R, \allowbreak \mathbf{S})$, where
    \begin{itemize}
        \item $V$ is a finite set of \textit{variables} (represented by boldface characters),
        \item $\mathcal{A}$ is a finite alphabet of \textit{terminals},
        \item $R$ is a finite set of \textit{production rules}, which are in the form $\mathbf{A} \to a\mathbf{B}$ or $\mathbf{A} \to \varepsilon$ for $\mathbf{A}, \mathbf{B} \in V$ and $a \in \mathcal{A}$,
        \item $\mathbf{S} \in V$ is the \textit{start variable}.
    \end{itemize}
    If $x \in \mathcal{A}^*$, $\mathbf{A} \in V$, and $\mathbf{A} \to \alpha$ is a production rule, we say $x\mathbf{A}$ \textit{yields} $x\alpha$, written $x\mathbf{A} \Rightarrow x\alpha$ (here, $\alpha$, $x\mathbf{A}$, and $x\alpha$ are viewed as strings over $V \sqcup \mathcal{A}$). If $\alpha, \beta \in (V \sqcup \mathcal{A})^*$, we say $\alpha$ \textit{derives} $\beta$ if there exists a sequence $\alpha_1, \dots, \alpha_n$ such that $\alpha \Rightarrow \alpha_1 \Rightarrow \dots \Rightarrow \alpha_n \Rightarrow \beta$. The \textit{language generated by $C$} is the set of all words in $\mathcal{A}^*$ that can be derived from the start symbol $\mathbf{S}$.
\end{defi}

\begin{rem}
    A more standard definition of a regular grammar allows for rules of the form $\mathbf{A} \to a$. The definition above is equivalent, and will be more convenient for our purposes.
\end{rem}

A well-known result is that regular grammars are equivalent to FSAs, in that they both characterize the class of regular languages. The following proposition provides an explicit description of the regular grammar that generates a given regular language.

\begin{prop}
    \label{proposition: grammar rules}
    Let $\mathcal{L}$ be a regular language, and let $\mathcal{D} = (Q, \mathcal{A}, \delta, \mathbf{q}_0, A)$ be a FSA accepting $\mathcal{L}$. Then $\mathcal{L}$ is generated by some regular grammar $C = (Q, \mathcal{A}, R, \mathbf{q}_0)$, consisting of the production rules
    \begin{itemize}
        \item $\mathbf{q} \to a\delta(\mathbf{q}, a)$ for all $\mathbf{q} \in Q$ and $a \in \mathcal{A}$,
        \item $\mathbf{q} \to \varepsilon$ for all $\mathbf{q} \in A$.
    \end{itemize}
\end{prop}
\noindent We will be most interested in the following generating function associated to a language.

\begin{defi}[Growth series]
    The \textit{growth series} of a language $\mathcal{L}$ over an alphabet $\mathcal{A}$ is the formal power series $\sum_{n = 0}^\infty |\mathcal{L} \cap \mathcal{A}^n|z^n$, where the coefficient of $z^n$ is the number of length $n$ words in $\mathcal{L}$.
\end{defi}

Sections \ref{sec: theorem 1} through \ref{sec: theorem 3} will be dedicated to studying the growth series associated to the \textit{geodesic language} of a group. Our primary method of computing these series will be the following important result. One of the original descriptions of this technique (in a far more general setting) is due to \cite{Cho63}; we include the proof for completeness.

\begin{thm}
    \label{theorem: chomsky black magic}
    Suppose that $C = (V, \mathcal{A}, R, \mathbf{A}_1)$ is a regular grammar with variables $V = \{\mathbf{A}_1, \dots, \mathbf{A}_n\}$ satisfying two conditions:
    \begin{itemize}
        \item The grammar $C$ is \textit{unambiguous}, in the sense that every word $x$ in the language generated by $C$ has a unique derivation from the start symbol $\mathbf{A}_1$.
        \item Every variable $\mathbf{A}_i$ is \textit{reachable} from the start symbol $\mathbf{A}_1$, in the sense that there is some word $\alpha \in (V \sqcup \mathcal{A})^*$ containing $\mathbf{A}_i$ that can be derived from $\mathbf{A}_1$.
    \end{itemize}
    For each $i = 1, \dots, n$, let $\mathcal{L}_i$ be the language consisting of all words that can be derived from $\mathbf{A}_i$, and let $A_i(z)$ be the growth series associated to $\mathcal{L}_i$. Then each series $A_i(z)$ satisfies the equation
    \[
        A_i(z) =
        \begin{cases}
            z\sum_{(\mathbf{A}_i \to a\mathbf{A}_j) \in R} A_j(z) + 1
             & \text{if $\mathbf{A}_i \to \varepsilon$ is a production rule in $R$,} \\
            z\sum_{(\mathbf{A}_i \to a\mathbf{A}_j) \in R} A_j(z)
             & \text{otherwise.}
        \end{cases}
    \]
\end{thm}
\begin{proof}
    Fix a variable $\mathbf{A}_i$. Let $c_{j, n} = |\mathcal{L}_j \cap \mathcal{A}^n|$ denote the $n$th coefficient of the power series $A_j(z)$. Comparing coefficients, we see that proving the theorem amounts to showing that
    \begin{equation}
        \label{eq: 1, theorem: chomsky black magic}
        c_{i, n}
        = \sum_{(\mathbf{A}_i \to a\mathbf{A}_j) \in R} c_{j, n - 1}
    \end{equation}
    for all $n \geq 1$, and
    \begin{equation}
        \label{eq: 2, theorem: chomsky black magic}
        c_{i, 0} =
        \begin{cases}
            1 & \text{if $\mathbf{A}_i \to \varepsilon$ is a production rule in $R$,} \\
            0 & \text{otherwise.}
        \end{cases}
    \end{equation}
    Equation (\ref{eq: 2, theorem: chomsky black magic}) is immediate: $\mathbf{A}_i \to \varepsilon$ is a production rule if and only if $\varepsilon$ (which is the only word of length $0$) is a word in $\mathcal{L}_i$.

    Next, we prove (\ref{eq: 1, theorem: chomsky black magic}) for $n \geq 1$. Consider a word $x \in \mathcal{L}_i \cap \mathcal{A}^n$. This word must be produced by some derivation
    \[
        \mathbf{A}_i \Rightarrow a\mathbf{A}_{j} \Rightarrow
        \dots \Rightarrow ay = x.
    \]
    Hence, there is some rule $\mathbf{A}_i \to a\mathbf{A}_j$ in $R$ such that $x = ay$ for $y \in \mathcal{L}_j \cap \mathcal{A}^{n - 1}$. This gives us a map
    \[
        \Phi : \mathcal{L}_i \cap \mathcal{A}^n
        \to \bigcup_{(\mathbf{A}_i \to a\mathbf{A}_j) \in R}
        \{a\} \times (\mathcal{L}_j \cap \mathcal{A}^{n - 1}),
        \quad \Phi(ay) = (a, y).
    \]
    Note that this map is a bijection, with inverse given by $(a, y) \mapsto ay$. We claim that the codomain of $\Phi$ is a disjoint union. Suppose toward a contradiction that $(a, y)$ were in both $\{a\} \times (\mathcal{L}_j \cap \mathcal{A}^{n - 1})$ and $\{a\} \times (\mathcal{L}_k \cap \mathcal{A}^{n - 1})$ for distinct $j, k$. Because $\mathbf{A}_i$ is reachable from $\mathbf{A}_1$, there exists a derivation
    \[
        \mathbf{A}_1
        \Rightarrow \dots \Rightarrow b_1 \cdots b_\ell \mathbf{A}_i.
    \]
    This derivation can then be continued as
    \[
        \mathbf{A}_1
        \Rightarrow \dots \Rightarrow b_1 \cdots b_\ell \mathbf{A}_i
        \Rightarrow b_1 \cdots b_\ell a \mathbf{A}_j
        \Rightarrow \dots \Rightarrow b_1 \cdots b_\ell ay
    \]
    or
    \[
        \mathbf{A}_1
        \Rightarrow \dots \Rightarrow b_1 \cdots b_\ell \mathbf{A}_i
        \Rightarrow b_1 \cdots b_\ell a \mathbf{A}_k
        \Rightarrow \dots \Rightarrow b_1 \cdots b_\ell ay.
    \]
    But this gives two distinct derivations of the word $b_1 \cdots b_\ell ay$, contradicting the unambiguity of $C$. Thus, the bijectivity of $\Phi$ implies
    \[
        c_{i, n}
        = |\mathcal{L}_i \cap \mathcal{A}^n|
        = \left|\bigcup_{(\mathbf{A}_i \to a\mathbf{A}_j) \in R}
        \{a\} \times (\mathcal{L}_j \cap \mathcal{A}^{n - 1})\right|
        = \sum_{(\mathbf{A}_i \to a\mathbf{A}_j) \in R} c_{j, n - 1},
    \]
    which proves (\ref{eq: 1, theorem: chomsky black magic}).
\end{proof}

\section{Numbered graph products}
\label{sec: ngp}
\begin{defi}[Graph products]
    Let $\Gamma$ be a simplicial graph, and let $G_v = \langle S_v \mid R_v \rangle$ be a group for all $v \in V\Gamma$. We define the \textit{graph product} of the groups $(G_v, S_v)$ over $\Gamma$ as the group with
    \begin{itemize}
        \item generators $\bigcup_{v \in V} S_v$,
        \item relations $\bigcup_{v \in V} R_v$ and $[s_u, s_v] = 1$ for all $\{u, v\} \in E$ and $s_u \in S_u, s_v \in S_v$.
    \end{itemize}
    Here, $[s_v, s_w]$ denotes the commutator $s_vs_ws_v^{-1}s_w^{-1}$.
\end{defi}

\begin{exa}
    When $\Gamma$ is totally disconnected, a graph product over $\Gamma$ is the same as a free product. When $\Gamma$ is complete, a graph product over $\Gamma$ is the same as a direct product. Thus, graph products generalize both these concepts.
\end{exa}

\begin{defi}[Numbered graph products]
    A \textit{numbered graph} $(\Gamma, N)$ consists of a simplicial graph $\Gamma$ and a map $N : V\Gamma \to \N$. We will refer to the numbers $N(v)$ for $v \in V\Gamma$ as the \textit{vertex numbers} of $(\Gamma, N)$. The \textit{numbered graph product} (NGP) associated to the numbered graph $(\Gamma, N)$, written $\NGP(\Gamma, N)$ or $\NGP(\Gamma)$, is the graph product of the cyclic groups $G_v = \langle v \mid v^{N(v)} = 1 \rangle$. We call the set $S = V\Gamma \cup (V\Gamma)^{-1}$ the \textit{standard generating set} of $\NGP(\Gamma)$.
\end{defi}

\begin{exa}
    Consider the following numbered graph $(\Gamma, N)$, where each vertex is labeled with its corresponding vertex number:
    \begin{center}
        \begin{tikzpicture}
            \path (0, 2) coordinate (v1);
            \path (2, 2) coordinate (v2);
            \path (2, 0) coordinate (v3);
            \path (0, 0) coordinate (v4);
            \draw[fill] (v1) circle (2pt);
            \draw[fill] (v2) circle (2pt);
            \draw[fill] (v3) circle (2pt);
            \draw[fill] (v4) circle (2pt);
            \draw (v1) node[above left] {$v_1, 20$} -- (v2) node[above right] {$v_2, 7$} -- (v3) node[below right] {$v_3, 2$} -- (v4) node[below left] {$v_4, 13$} -- cycle;
        \end{tikzpicture}
    \end{center}
    The NGP associated to $(\Gamma, N)$ is the group
    \[
        \langle v_1, v_2, v_3, v_4
        \mid v_1^{20} = v_2^7 = v_3^2 = v_4^{13} = 1,
        [v_1, v_2] = [v_2, v_3] = [v_3, v_4] = [v_4, v_1] = 1 \rangle.
    \]
    The standard generating set of $\NGP(\Gamma)$ is $S = \{v_1, v_1^{-1}, v_2, v_2^{-1}, v_3, v_4, v_4^{-1}\}$. Note that $v_3$ and $v_3^{-1}$ are equal, and thus are not distinct elements in $S$.

    We can write this particular NGP in a simpler form. Given a word $x \in S^*$, we may use the commuting relations to shuffle the letters $v_1^{\pm 1}, v_3$ to the left and the letters $v_2^{\pm 1}, v_4^{\pm 1}$ to the right. This allows us to rewrite the group element represented by $x$ as the concatenation of a word $y \in \langle v_1, v_3\rangle$ and a word $z \in \langle v_2, v_4\rangle$. There are no relations between $v_1$ and $v_3$, so they generate a free product; similarly for $v_2$ and $v_4$. Hence, we have $\NGP(\Gamma) \cong (\Z/20\Z * \Z/2\Z) \oplus (\Z/7\Z * \Z/13\Z)$.
\end{exa}

\begin{defi}[Geodesic languages]
    Let $G$ be a group with a symmetric generating set $S$. There is a natural evaluation map $\mathrm{ev}_{(G, S)} : S^* \to G$ that sends a word $s_1 \cdots s_n \in S^*$ to the product $s_1 \cdots s_n \in G$. We say $x \in S^*$ \textit{evaluates in $G$} to $\mathrm{ev}_{(G, S)}(x)$. A word $x \in S^*$ is called a \textit{geodesic} of $(G, S)$ if there is no shorter word $y \in S^*$ that evaluates to the same element as $x$. The set of all geodesics is called the \textit{geodesic language} of $(G, S)$. The \textit{geodesic growth series} of $(G, S)$ is the growth series of the geodesic language of $(G, S)$.
\end{defi}

\begin{rem}
    A more common interpretation of a geodesic is as the shortest \textit{path} in the \textit{Cayley graph} of a group $(G, S)$. These paths are in bijection with words over the symmetric generating set $S \cup S^{-1}$. For this reason, it is crucial in the definition of a geodesic that the generating set $S$ is symmetric. This leads to minor inconveniences in the case of numbered graph products since the natural choice of symmetric generating set for a cyclic group of order $2$ contains a single element, whereas for a cyclic group of any other order the symmetric generating set contains two elements. This discrepancy will lead to important edge cases in the results to follow.
\end{rem}

Note that whenever we speak of the geodesics (or geodesic growth) of some numbered graph product $\NGP(\Gamma)$, we are always working with respect to the standard generating set of $\NGP(\Gamma)$.

In the remainder of this section, our goal will be to construct a FSA that recognizes the geodesic language of an NGP. It is already known that this language is regular; the paper \cite{Loe02} constructs a FSA that recognizes the geodesic language of any graph product whose vertex groups have regular geodesic languages. To do this, they modify the FSA's associated to individual vertex groups, then show that the \textit{product} of these FSA's recognize the geodesic language of the graph product. We will follow their construction, but we will interpret the product FSA in a setting specific to NGPs. In particular, the states will be objects called \textit{powered cliques}.

\begin{defi}[Cliques and Links]
    \label{def: cliques and links}
    Let $\Gamma$ be a simplicial graph. A subset $\sigma \subset V\Gamma$ is called a \textit{clique} if the vertices of $\sigma$ span a complete subgraph of $\Gamma$. If $\sigma$ has $i$ elements, we call $\sigma$ an \textit{$i$-clique}. A subset $\tau \subseteq \sigma$ is called a \textit{subclique} of $\sigma$. We define the \textit{link} of a clique to be the set
    \[
        \Lk(\sigma) = \{ v \in V\Gamma \setminus \sigma \mid \{v\} \cup \sigma \text{ is a clique}\}
    \]
    When $\sigma$ is a singleton, we write $\Lk(v)$ instead of $\Lk(\{v\})$.
\end{defi}

\begin{defi}[Powered cliques]
\label{def: powered clique}
    A \textit{powered clique} of the numbered graph $(\Gamma, N)$ is a map $\Sigma: V\Gamma \rightarrow \Z$ satisfying the conditions:
    \begin{itemize}
        \item The \textit{support} of $\Sigma$, defined by $\supp \Sigma \coloneqq \{v \in V\Gamma \mid \Sigma(v) \neq 0\}$, is a clique.
        \item For $v \in V\Gamma$, $\Sigma(v)$ satisfies the bounds
              \[
                  \begin{cases}
                      0 \leq \Sigma(v) \leq 1
                       & \text{if $N(v) = 2$,} \\
                      -\lfloor N(v) / 2 \rfloor \leq \Sigma(v)
                      \leq \lfloor N(v) / 2 \rfloor
                       & \text{otherwise,}     \\
                  \end{cases}
              \]
              where $\lfloor \cdot \rfloor$ is the floor function.
    \end{itemize}
    We refer to the number $\Sigma(v)$ as the \textit{power of $\Sigma$ at the vertex $v$}. We will often define a powered clique by first specifying its support, then defining its power at vertices in its support. The empty powered clique is the map $\Sigma : V\Gamma \rightarrow \Z$ that has empty support.
\end{defi}

\begin{thm}
    \label{theorem: fsa}
    Let $(\Gamma, N)$ be a numbered graph. The geodesic language of $\NGP(\Gamma)$ is accepted by the FSA $\mathcal{D}$, where:
    \begin{itemize}
        \item The accept states are the powered cliques of $(\Gamma, N)$, and there is a single reject state $q_\mathrm{rej} \in Q \setminus A$. The start state is the empty powered clique.
        \item The transition map $\delta$ is defined as follows. Given a vertex $v \in V\Gamma$, if $\Sigma$ is a powered clique such that $0 \leq \Sigma(v) < \lfloor N(v) / 2\rfloor$, we define the transition $\delta(\Sigma, v)$ to be the powered clique with support $(\supp(\Sigma) \cap \allowbreak \Lk(v)) \cup \{v\}$ and powers
              \[
                  \delta(\Sigma, v)(u) =
                  \begin{cases}
                      \Sigma(u)     & \text{if $u \in \Lk(v)$,} \\
                      \Sigma(v) + 1 & \text{if $u = v$.}
                  \end{cases}
              \]
              Symmetrically, if $N(v) \neq 2$ and $\Sigma$ is a powered clique such that $-\lfloor N(v) / 2 \rfloor < \Sigma \leq 0$, we define the transition $\delta(\Sigma, v^{-1})$ to be the powered clique with support $(\supp(\Sigma) \cap \allowbreak \Lk(v)) \cup \{v\}$ and powers
              \[
                  \delta(\Sigma, v^{-1})(u) =
                  \begin{cases}
                      \Sigma(u)     & \text{if $u \in \Lk(v)$,} \\
                      \Sigma(v) - 1 & \text{if $u = v$.}
                  \end{cases}
              \]
              In all other cases, we define the transition $\delta(\Sigma, v^{\pm 1})$ to be $q_\mathrm{fail}$.
    \end{itemize}
\end{thm}

To prove Theorem \ref{theorem: fsa}, we will need the following fact from \cite{Loe02} about the geodesic language of graph products. A rigorous proof of this fact involves developing a \textit{normal form} for graph products, see \cite{Gre90} for the construction of such a normal form.

\begin{lem}
    \label{lemma: word problem}
    Let $G$ be a graph product of the groups $(G_v, S_v)$, and let $S = \bigcup_{v \in V} S_v$. Given a word $s_1 \cdots s_n \in S^*$, we define the following transformation, called \textit{shuffling}:
    \begin{enumerate}
        \item[(S)] Interchange two consecutive letters $s_i \in S_u$ and $s_{i + 1} \in S_v$ if $\{u, v\}$ is an edge in $\Gamma$.
    \end{enumerate}
    Then a word $x \in S^*$ is \textit{not} a geodesic of $(G, S)$ if and only if there is a way to shuffle the letters of $x$ to form a subword in $S_v^*$ that is \textit{not} a geodesic of $(G_v, S_v)$.
\end{lem}

\begin{proof}[Proof of Theorem \ref{theorem: fsa}]
    We write $V\Gamma = \{v_1, \dots, v_n\}$, $G_i \coloneqq \langle v_i\rangle$, and $S_i = \{v_i, v_i^{-1}\}$. Fix an index $i$, and let $m \coloneqq \lfloor N(v_i) / 2\rfloor$. The geodesic language of the cyclic group $G_i$ is the set of words
    \[
        \mathcal{L}_i
        = \{v_i^{-m}, \dots, v_i^{-1}, \varepsilon, v_i, \dots, v_i^m\}
    \]
    We construct a FSA $\mathcal{D}_i$ that accepts $\mathcal{L}_i$ as follows. The accept states of $\mathcal{D}_i$ are the group elements $\{v_i^{-m}, \dots, v_i^m\} \subseteq G_i$, there is a single reject state $q_\mathrm{rej}$, and the start state is the identity. For each $0 \leq k < m$, there is a $v_i$-edge from $v_i^k$ to $v_i^{k + 1}$ and a $v_i^{-1}$-edge from $v_i^{-k}$ to $v_i^{-k - 1}$. Clearly, $\mathcal{D}_i$ accepts the geodesic language of $G_i$.

    Next, we modify the FSA $\mathcal{D}_i$ to read words over the alphabet $S$. The modified FSA, which we call $\widetilde{\mathcal{D}}_i$, is the same as $\mathcal{D}_i$, but has the following additional transitions:
    \begin{itemize}
        \item If $\{v_i, v_j\}$ is an edge in $\Gamma$, add a $v_j$-transition and a $v_j^{-1}$-transition from $v_i^k$ to itself for every state $v_i^k$.
        \item If $\{v_i, v_j\}$ is not an edge in $\Gamma$, add a $v_j$-transition and a $v_j^{-1}$-transition from $v_i^k$ to $1$ for every state $v_i^k$.
    \end{itemize}
    The FSA $\widetilde{\mathcal{D}}_i$ pictured in Figure \ref{fig:FSAmodified}. Observe that $\widetilde{\mathcal{D}}_i$ rejects a word $x \in S^*$ if and only if there is a way to shuffle the letters of $x$ (as defined by Lemma \ref{lemma: word problem}) to form a subword in $S_i^*$ that is not a geodesic of $(G_i, S_i)$.

    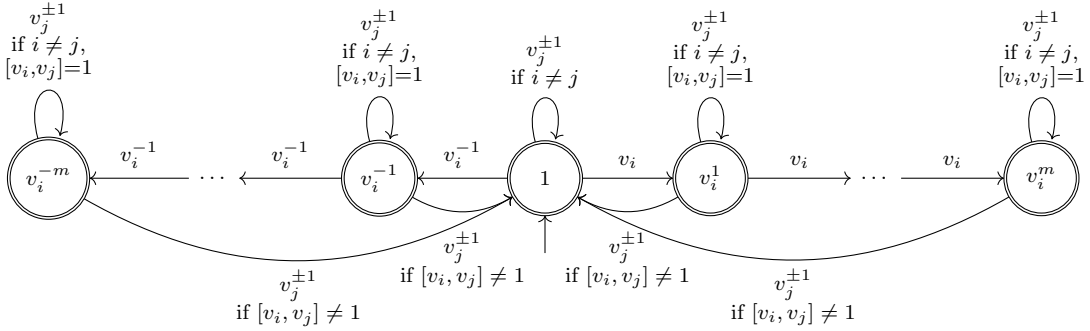
\begin{figure}
        \centering
        \begin{tikzpicture}
            \begin{scope}[font={\scriptsize}]
                \node[state, accepting] (nD) at (-6.6, 0) {$v_i^{-m}$};
                \node (ndots) at (-4.4, 0) {\dots};
                \node[state, accepting] (n1) at (-2.2, 0) {$v_i^{-1}$};
                \node[state, accepting, initial below, initial text={}] (start) at (0, 0) {$1$};
                \node[state, accepting] (p1) at (2.2, 0) {$v_i^1$};
                \node (dots) at (4.4, 0) {\dots};
                \node[state, accepting] (pD) at (6.6, 0) {$v_i^m$};

                \path[->] (ndots) edge node[above] {$v_i^{-1}$} (nD)
                (n1)    edge node[above] {$v_i^{-1}$} (ndots)
                (start) edge node[above] {$v_i^{-1}$} (n1)
                (start) edge node[above] {$v_i$}      (p1)
                (p1)    edge node[above] {$v_i$}      (dots)
                (dots)  edge node[above] {$v_i$}      (pD);
            \end{scope}

            \path[->] (nD)    edge[loop above]
            node[above]
            {$\substack{
                v_j^{\pm 1} \\
                \text{if $i \neq j$,}\\
                [v_i, v_j] = 1               
                }$} (nD)
            (n1)    edge[loop above]
            node[above]
            {$\substack{
                v_j^{\pm 1} \\
                \text{if $i \neq j$,}\\
                [v_i,v_j] = 1
                }$} (n1)
            (start) edge[loop above]
            node[above]
            {$\substack{
                v_j^{\pm 1} \\
                \text{if $i \neq j$}
                }$} (start)
            (p1)    edge[loop above]
            node[above]
            {$\substack{
                v_j^{\pm 1} \\
                \text{if $i \neq j$,}\\
                [v_i,v_j] = 1
                }$} (p1)
            (pD)    edge[loop above]
            node[above]
            {$\substack{
                v_j^{\pm 1} \\
                \text{if $i \neq j$,}\\
                [v_i, v_j] = 1
                }$} (pD);

            \path[->] (nD) edge[bend right]
            node[below]
            {$\substack{
                v_j^{\pm 1} \\
                \text{if $[v_i, v_j] \neq 1$}
                }$} (start)
            (n1) edge[bend right]
            node[below, yshift=-0.15cm]
            {$\substack{
                v_j^{\pm 1} \\
                \text{if $[v_i, v_j] \neq 1$}
                }$} (start)
            (p1) edge[bend left]
            node[below, yshift=-0.15cm]
            {$\substack{
                v_j^{\pm 1} \\
                \text{if $[v_i, v_j] \neq 1$}
                }$} (start)
            (pD) edge[bend left]
            node[below]
            {$\substack{
                v_j^{\pm 1} \\
                \text{if $[v_i, v_j] \neq 1$}
                }$} (start);
        \end{tikzpicture}
        \caption{The FSA $\widetilde{\mathcal{D}}_i$ when $N(v_i) \neq 2$.}
        \label{fig:FSAmodified}
    \end{figure}

    Next, we construct the \textit{product} $\mathcal{D}$ of the modified FSA's $\widetilde{\mathcal{D}}_i$. The accept states of $\mathcal{D}$ are tuples $(v_1^{k_1}, \dots, v_n^{k_n})$, where each $v_i^{k_i}$ is a state of $\widetilde{\mathcal{D}}_i$. There is a single reject state $q_\mathrm{rej}$, and the tuple $(1, \dots, 1)$ is designated as the start state. Given a generator $s \in S$, there is an $s$-edge from $(v_1^{k_1}, \dots, v_n^{k_n})$ to $(v_1^{\ell_1}, \dots, v_n^{\ell_n})$ if there is an $s$-edge from $v_i^{k_i}$ to $v_i^{\ell_i}$ in each $\widetilde{\mathcal{D}}_i$. If there is an $s$-edge from $v_i^{k_i}$ to $q_\mathrm{rej}$ in any $\widetilde{\mathcal{D}}_i$, then there is an $s$-edge from $(v_1^{k_1}, \dots, v_n^{k_n})$ to $q_\mathrm{rej}$ in $\mathcal{D}$. By construction, $\mathcal{D}$ rejects a word in $S^*$ if and only if the word is rejected by some individual FSA $\widetilde{\mathcal{D}}_i$. Therefore, $\mathcal{D}$ rejects a word in $x \in S^*$ if and only if there is a way to shuffle the letters of $x$ to form a subword in any $S_i^*$ that is not a geodesic of $(G_i, S_i)$. By Lemma \ref{lemma: word problem}, $\mathcal{D}$ accepts the geodesic language of $(G, S)$.

    We claim that the FSA $\mathcal{D}$ is exactly the one described by the theorem statement. Observe that the states of $\mathcal{D}$ are in bijection with powered cliques, where to each state $(v_1^{k_1}, \dots, v_n^{k_n})$ we associate the powered clique $v_i \mapsto k_i$, and to each powered clique $\Sigma$ we associate the state $(v_1^{\Sigma(v_1)}, \dots, v_n^{\Sigma(v_n)})$. Thus, we may view the states of $\mathcal{D}$ as powered cliques.

    It remains to show that the transition map of $\mathcal{D}$ is given by $\delta$. Let $\Sigma$ be a powered clique, which corresponds to the state $(v_1^{\Sigma(v_1)}, \dots, v_n^{\Sigma(v_n)})$. Let $v_i \in V\Gamma$. The FSA $\widetilde{\mathcal{D}}_i$ has a $v_i$-transition given by
    \[
        \begin{cases}
            v_i^{\Sigma(v_i)} \to v_i^{\Sigma(v_i) + 1}
             & \text{if $0 \leq \Sigma(v_i) < \lfloor N(v_i) / 2\rfloor$,} \\
            v_i^{\Sigma(v_i)} \to q_\mathrm{rej}
             & \text{otherwise.}
        \end{cases}
    \]
    The other FSA's $\widetilde{\mathcal{D}}_j$ have a $v_i$-transition given by
    \[
        \begin{cases}
            v_j^{\Sigma(v_j)} \to v_j^{\Sigma(v_j)}
             & \text{if $v_j \in \Lk(v_i)$,} \\
            v_j^{\Sigma(v_j)} \to 1
             & \text{otherwise.}
        \end{cases}
    \]
    Therefore, if $0 \leq \Sigma(v_i) < \lfloor N(v_i) / 2\rfloor$, $\mathcal{D}$ has a $v_i$-transition from $\Sigma$ to the powered clique with support $(\supp(\Sigma) \cap \Lk(v_i)) \cup \{v_i\}$ and the same powers as $\Sigma$ except with the power at $v_i$ incremented. Otherwise, $\mathcal{D}$ has a $v_i$-transition from $\Sigma$ to the fail state. In particular, the transitions along $v_i$ are precisely those prescribed by $\delta$. If we instead consider transitions along $v_i^{-1}$, we have a symmetric situation, excluding the case where $N(v_i) = 2$. This completes the proof.
\end{proof}

\section{Link-regular NGPs}
\label{sec: theorem 1}
Now that we have introduced the necessary background, we are prepared to explore the geodesic growth of NGPs. In particular, we will focus on NGPs associated to numbered graphs that obey a strong regularity condition called \textit{link-regularity}. Our main result in this section will show that NGPs associated to link-regular numbered graphs which are \textit{equivalent} (in a certain sense, which we will define) have the same geodesic growth series. This result then allows us to provide examples of NGPs associated to non-isomorphic graphs that have the same geodesic growth series.

We begin by defining \textit{link-regular} numbered graphs and the \textit{equivalence} of such graphs. Note that when the vertex numbers $N(v)$ are all $2$, our definition of link-regularity (Definition \ref{def: link regularity}) is equivalent to Definition 3.1 in \cite{Ant12}. We will discuss this case at length in Section \ref{sec: theorem 2}.

\begin{defi}[Multisets]
    A \textit{multiset} $A$ consists of an underlying set $\supp(A)$ (called the \textit{support} of $A$) and a map $m_A : \supp(A) \to \N$ (where for each $a \in \supp(A)$, $m_A(a)$ is called the \textit{multiplicity} of $a$). A multiset $B$ is a \textit{subset} of $A$, written $B \subseteq A$, if $\supp(B) \subseteq \supp(A)$ and $m_B(b) \leq m_A(b)$ for all $b \in \supp(B)$.

    The \textit{sum} of two multisets $A$ and $B$ is the multiset $A + B$ where $\supp(A+B) = A \cup B$ and $m_{A+B} = m_A + m_B$. The \textit{difference} of two multisets $A$ and $B$ is the multiset $A - B$ where $\supp(A-B) = \{a \in A \mid m_A(a) > m_B(a)\}$ and $m_{A-B} = m_A - m_B$.
\end{defi}

\begin{rem}
    Informally, we will write multisets as sets with repeated elements, e.g., $\{a, a, b, c, c, c\}$. Similarly, we will often define multisets with set-builder notation, e.g., $\{\lfloor n / 2 \rfloor \mid n \in \N\}$ is the multi-set $\{0, 1, 1, 2, 2, 3, 3, \dots\}$.
\end{rem}

\begin{defi}[Link-regularity]
    \label{def: link regularity}
    Let $(\Gamma, N)$ be a numbered graph. If $U \subseteq V\Gamma$, we let $N(U)$ denote the multiset $\{N(v) \mid v \in U\}$. We say $\Gamma$ is \textit{link-regular} if, whenever $N(\sigma_1) = N(\sigma_2)$ for cliques $\sigma_1,\sigma_2$, we have $N(\Lk(\sigma_1)) = N(\Lk(\sigma_2))$.
\end{defi}

\begin{exa}
    Let $(\Gamma, N)$ be the numbered graph pictured below:
    \begin{center}
        \begin{tikzpicture}
            \foreach \th/\lab in {1/3, 2/7, 3/3, 4/7, 5/3, 6/7} {
                    \node[label={\th * 360 / 6:$\lab$}] (v\th) at (\th * 360 / 6:1.5) {};
                    \filldraw (v\th) circle (2pt);
                }
            \draw (v1.center) -- (v2.center) -- (v3.center) -- (v4.center) -- (v5.center) -- (v6.center) -- cycle;
            \draw (v1.center) -- (v4.center);
            \draw (v2.center) -- (v5.center);
            \draw (v3.center) -- (v6.center);
        \end{tikzpicture}
    \end{center}
    For cliques $\sigma$ where $N(\sigma_1) = \{3\}$, we have $N(\Lk(\sigma_1))= \{7,7,7\}$. Similarly, for cliques $\sigma$ with $N(\sigma) = \{7\}$, we have $N(\Lk(\sigma))= \{3,3,3\}$. The link of all $2$-cliques is empty, and there are no cliques of size greater than $2$. Therefore, $(\Gamma, N)$ is link-regular.
\end{exa}

\begin{defi}[Equivalence of link-regular graphs]
\label{def: graph equivalence}
    We say two link-regular numbered graphs $(\Gamma, N)$ and $(\Gamma', N')$ are \textit{equivalent} if:
    \begin{enumerate}
        \item[(i)] $N(V\Gamma) = N'(V\Gamma')$.
        \item[(ii)] Whenever cliques $\sigma \subseteq V\Gamma$ and $\sigma' \subseteq V\Gamma'$ satisfy $N(\sigma) = N'(\sigma')$, we have $N(\Lk(\sigma)) = N'(\Lk(\sigma'))$.
    \end{enumerate}
    Note that condition (ii) is equivalent to requiring the disjoint union $(\Gamma \sqcup \Gamma', N \sqcup N')$ (where $(N \sqcup N')(v) = N(v)$ for $v \in V\Gamma$ and $(N \sqcup N')(v) = N'(v)$ for $v \in V\Gamma'$) to be link-regular.
\end{defi}

The equivalence of link-regular numbered graphs is less strict than graph isomorphism, as we showcase in the following example.

\begin{exa}
    \label{example: equivalence}
    Let $(\Gamma, N)$ be the following link-regular numbered graph of two disjoint squares, and let $(\Gamma', N')$ be the following link-regular numbered graph of an octagon:
    \begin{center}
        \begin{tikzpicture}
            \foreach \th/\lab in {1/4, 2/6, 3/4, 4/6} {
                    \node[label={[label distance=-0.1cm]\th * 360 / 4 + 360 / 8:$\lab$}] (v\th) at (\th * 360 / 4 + 360 / 8:1) {};
                    \filldraw (v\th) circle (2pt);
                }
            \draw (v1.center) -- (v2.center) -- (v3.center) -- (v4.center) -- cycle;

            \foreach \th/\lab in {1/4, 2/6, 3/4, 4/6} {
                    \node[label={[label distance=-0.1cm]\th * 360 / 4 + 360 / 8:$\lab$}] (v\th) at ($(3, 0) + (\th * 360 / 4 + 360 / 8:1)$) {};
                    \filldraw (v\th) circle (2pt);
                }
            \draw (v1.center) -- (v2.center) -- (v3.center) -- (v4.center) -- cycle;

            \foreach \th/\lab in {1/4, 2/6, 3/4, 4/6, 5/4, 6/6, 7/4, 8/6} {
                    \node[label={[label distance=-0.1cm]\th * 360 / 8 + 360 / 16:$\lab$}] (v\th) at ($(7.5, 0) + (\th * 360 / 8 + 360 / 16:1.5)$) {};
                    \filldraw (v\th) circle (2pt);
                }
            \draw (v1.center) -- (v2.center) -- (v3.center) -- (v4.center) -- (v5.center) -- (v6.center) -- (v7.center) -- (v8.center) -- cycle;
        \end{tikzpicture}
    \end{center}
    Evidently, these graphs are not isomorphic. However, the disjoint union of the graphs $(\Gamma, N)$ and $(\Gamma', N')$ is link-regular, since whenever $N(\sigma) = N'(\sigma') = \{4\}$, we have $N(\Lk(\sigma)) = N'(\Lk(\sigma')) = \{6,6\}$; whenever $N(\sigma) = N'(\sigma') = \{6\}$, we have $N(\Lk(\sigma)) = N'(\Lk(\sigma')) = \{4,4\}$; and whenever $N(\sigma) = N'(\sigma') = \{4, 6\}$, we have $N(\Lk(\sigma)) = N'(\Lk(\sigma')) = \emptyset$. Since we also have $N(V\Gamma) = \{4,4,4,4,6,6,6,6\} = N'(V\Gamma')$, $(\Gamma, N)$ and $(\Gamma', N')$ are equivalent link-regular numbered graphs.
\end{exa}

Our goal is to prove that the NGPs associated to equivalent link-regular numbered graphs $(\Gamma, N)$ and $(\Gamma', N')$ have the same geodesic growth series. To do this, we will study the FSAs $\mathcal{D}$ and $\mathcal{D}'$ (which were constructed in Theorem \ref{theorem: fsa}) that recognize the geodesic languages of $\NGP(\Gamma)$ and $\NGP(\Gamma')$, respectively. Because the transitions of $\mathcal{D}$ and $\mathcal{D}'$ are defined in terms of cliques and links, we will need several results comparing the cliques and links of $\Gamma$ and $\Gamma'$.

\begin{defi}
    Given a clique $\sigma$ of a simplicial graph $\Gamma$ and a subclique $\tau \subseteq \sigma$, we define
    \[
        \Lk(\sigma; \tau)
        = \{v \in V\Gamma \setminus \sigma \mid \Lk(v) \cap \sigma = \tau\}.
    \]
\end{defi}

\begin{lem}
    \label{lemma: strengthening link regularity}
    Let $(\Gamma, N)$ be a link-regular numbered graph. Suppose $\sigma$ and $\sigma'$ are cliques of $\Gamma$ such that $N(\sigma) = N(\sigma')$. Then for any $\tau \subseteq \sigma$ and $\tau' \subseteq \sigma'$ such that $N(\tau) = N'(\tau')$, we have $N(\Lk(\sigma ; \tau)) = N(\Lk(\sigma'; \tau'))$.
\end{lem}
\begin{proof}
    First, we show that if $\sigma$ is a clique of $\Gamma$ and $\tau \subseteq \sigma$ is a subclique, then
    \begin{equation}
        \label{eq: link sigma tau}
        \Lk(\sigma ; \tau) = \Lk(\tau) \setminus \left((\sigma \setminus \tau) \sqcup \bigsqcup_{\tau \subsetneq \pi \subseteq \sigma} \Lk(\sigma; \pi)\right).
    \end{equation}
    If $v \in \Lk(\sigma ; \tau)$, we have by definition that $v \not\in \sigma$ and that $\Lk(v) \cap \sigma = \tau$. This implies $v \in \Lk(\tau)$, as well as $v \not\in \Lk(\sigma; \pi)$ for all $\pi \supsetneq \tau$. Hence, $v$ lies in the right-hand side of (\ref{eq: link sigma tau}). Conversely, suppose $v$ lies in the right-hand side of (\ref{eq: link sigma tau}). The condition that $v \in \Lk(\tau)$ implies $v \not\in \tau$ and $\tau \subseteq \Lk(v) \cap \sigma$. The exclusion of $v$ from $\sigma \setminus \tau$ further implies $v \not\in \sigma$. Finally, the exclusion of $v$ from the union $\bigsqcup_{\tau \subsetneq \pi \subseteq \sigma} \Lk_\pi(\sigma)$ means that $\Lk(v) \cap \sigma$ cannot contain $\pi$ for any $\pi \supsetneq \tau$, and so $\Lk(v) \cap \sigma = \tau$. Hence, $v \in \Lk(\sigma ; \tau)$.

    Now, let $\sigma, \sigma', \tau, \tau'$ be as in the statement of the lemma. Our proof will proceed by induction on the size of the subcliques $\tau$ and $\tau'$ (which must be of the same size since $N(\tau) = N(\tau')$). First, consider the base case when $|\tau| = |\sigma|$ and $|\tau'| = |\sigma'|$. Then $\tau = \sigma$ and $\tau' = \sigma'$, which implies $\Lk(\sigma; \tau) = \Lk(\sigma)$ and $\Lk(\sigma'; \tau') = \Lk(\sigma')$. By the link-regularity of $(\Gamma, N)$, we have
    \[
        N(\Lk(\sigma; \tau))
        = N(\Lk(\sigma))
        = N(\Lk(\sigma'))
        = N(\Lk(\sigma'; \tau')),
    \]
    as desired.

    Inductively, assume that $N(\Lk(\sigma; \pi)) = N(\Lk(\sigma; \pi'))$ for all subcliques $\pi \subseteq \sigma$ and $\pi' \subseteq \sigma'$ of size strictly larger than $\tau$ and $\tau'$ that satisfy $N(\pi) = N(\pi')$. By (\ref{eq: link sigma tau}), we have the identity
    \begin{align*}
        N(\Lk(\sigma ; \tau)) & = N(\Lk(\tau)) - \left(N(\sigma) - N(\tau) + \sum_{\tau \subsetneq \pi \subseteq \sigma} N(\Lk(\sigma; \pi))\right),
    \end{align*}
    as well as an analagous identity for $N(\Lk(\sigma'; \tau'))$. Because $N(\sigma) = N(\sigma')$ and $N(\tau) = N(\tau')$, there is a bijection $\sigma \rightarrow \sigma'$ that sends $\tau$ to $\tau'$ and preserves vertex numbers. Given a subclique $\tau \subsetneq \pi \subseteq \sigma$, this bijection determines a corresponding subclique $\tau' \subsetneq \pi' \subseteq \sigma'$ such that $N(\pi') = N(\pi)$. By the inductive hypothesis and the link-regularity of $(\Gamma, N)$,
    \begin{align*}
        N(\Lk(\sigma ; \tau))
         & = N(\Lk(\tau)) - \left(N(\sigma) - N(\tau) + \sum_{\tau \subsetneq \pi \subseteq \sigma} N(\Lk(\sigma; \pi))\right)      \\
         & = N(\Lk(\tau')) - \left(N(\sigma') - N(\tau') + \sum_{\tau' \subset \pi' \subseteq \sigma'} N(\Lk(\sigma'; \pi'))\right) \\
         & = N(\Lk(\sigma'; \tau')).
    \end{align*}
    This proves the claim for $\tau$ and $\tau'$, as desired.
\end{proof}

\begin{cor}
    \label{corollary: strengthening link regularity}
    Let $(\Gamma, N)$ and $(\Gamma', N')$ be equivalent link-regular numbered graphs. Suppose $\sigma \subseteq V\Gamma$ and $\sigma' \subseteq V\Gamma'$ are cliques such that $N(\sigma) = N'(\sigma')$. Then for any $\tau \subseteq \sigma$ and $\tau' \subseteq \sigma'$ such that $N(\tau) = N'(\tau')$, we have $N(\Lk(\sigma ; \tau)) = N'(\Lk(\sigma'; \tau'))$.
\end{cor}
\begin{proof}
    By definition, the disjoint union $(\Gamma \sqcup \Gamma', N \sqcup N')$ is link-regular. Now, $\sigma$ and $\sigma'$ are subcliques of $(\Gamma \sqcup \Gamma', N \sqcup N')$ and so the claim follows by Lemma \ref{lemma: strengthening link regularity}.
\end{proof}

\begin{defi}
\label{def: power profile}
    Let $\Sigma : V\Gamma \to \Z$ be a powered clique of a numbered graph $(\Gamma, N)$. The \textit{power profile of $\Sigma$}, written $P(\Sigma)$, is the multiset $\{(\Sigma(v), N(v)) \mid v \in \supp \Sigma\}$. If the power profile of $\Sigma$ is $P$, we say $\Sigma$ is a $P$-state.
\end{defi}

\begin{lem}
    \label{lemma: bijection between transitions}
    Let $(\Gamma, N)$ and $(\Gamma', N')$ be equivalent link-regular numbered graphs, and let $\Sigma : V\Gamma \to \Z$ and $\Sigma' : V\Gamma' \to \Z$ be powered cliques such that $P(\Sigma) = P(\Sigma')$. Then there is a bijection $\Phi : V\Gamma \to V\Gamma'$ such that
    \begin{equation}
        \label{eq: power profile of transition}
        P(\delta(\Sigma, v)) = P(\delta'(\Sigma', \Phi(v)))
        \quad \text{and} \quad
        P(\delta(\Sigma, v^{-1})) = P(\delta'(\Sigma', \Phi(v)^{-1}))
    \end{equation}
    for all $v \in V\Gamma$. Here, $\delta$ and $\delta'$ denote the transition maps of $\mathcal{D}$ and $\mathcal{D}'$, respectively.
\end{lem}
\begin{proof}
    Let $\sigma$ and $\sigma'$ denote the supports of $\Sigma$ and $\Sigma'$, respectively. Because $P(\Sigma) = P(\Sigma')$, we can find a bijection $\Phi : \sigma \to \sigma'$ such that
    \begin{equation}
        \label{eq: properties of phi}
        N'(\Phi(v)) = N(v) \quad \text{and} \quad
        \Sigma'(\Phi(v)) = \Sigma(v)
    \end{equation}
    for all $v \in \sigma$. By the definition of the transition maps $\delta$ and $\delta'$, $\Phi$ satisfies (\ref{eq: power profile of transition}) for all $v \in \sigma$.

    We aim to extend $\Phi$ to a bijection $V\Gamma \to V\Gamma'$. Observe that
    \[
        V\Gamma = \sigma \sqcup \bigsqcup_{\tau \subseteq \sigma} \Lk(\sigma; \tau)
    \]
    and
    \[
        V\Gamma'
        = \sigma' \sqcup \bigsqcup_{\tau' \subseteq \sigma'} \Lk(\sigma'; \tau')
        = \sigma' \sqcup \bigsqcup_{\tau \subseteq \sigma} \Lk(\sigma'; \Phi(\tau)).
    \]
    To extend $\Phi$ to all of $V\Gamma$, it suffices to construct individual bijections $\Phi : \Lk(\sigma; \tau) \to \Lk(\sigma'; \Phi(\tau))$ for all $\tau \subseteq \sigma$ satisfying (\ref{eq: power profile of transition}). Piecing these bijections together into a single bijection $V\Gamma \to V\Gamma'$ yields the desired map.

    Let $\tau \subseteq \sigma$. By (\ref{eq: properties of phi}), we have $N(\Phi(\tau)) = N(\tau)$, so we have by Corollary \ref{corollary: strengthening link regularity} that $N(\Lk(\sigma; \tau)) = N(\Lk(\sigma'; \Phi(\tau))$. Hence, we can find a bijection $\Phi : \Lk(\sigma; \tau) \to \Lk(\sigma'; \Phi(\tau))$ such that
    \[
        N'(\Phi(v)) = N(v)
    \]
    for all $v \in \Lk(\sigma; \tau)$. We claim that this map satisfies (\ref{eq: power profile of transition}). For $v \in \Lk(\sigma; \tau)$, we have by definition of $\delta$ that
    \[
        P(\delta(\Sigma, v))
        = \{(1, N(v)))\} \sqcup \bigsqcup_{u \in \tau} \{(\Sigma(u), N(u))\},
    \]
    and because $\Phi(v) \in \Lk(\sigma'; \Phi(\tau))$, we similarly have
    \[
        P(\delta'(\Sigma', \Phi(v)))
        = \{(1, N'(\Phi(v)))\} \sqcup \bigsqcup_{u \in \Phi(\tau)} \{(\Sigma'(u), N'(u))\}.
    \]
    Applying (\ref{eq: properties of phi}), we get
    \begin{align*}
        P(\delta'(\Sigma', \Phi(v)))
         & = \{(1, N'(\Phi(v)))\} \sqcup \bigsqcup_{u \in \Phi(\tau)} \{(\Sigma'(u), N'(u))\}       \\
         & = \{(1, N'(\Phi(v)))\} \sqcup \bigsqcup_{u \in \tau} \{(\Sigma'(\Phi(u)), N'(\Phi(u)))\} \\
         & = \{(1, N(v)))\} \sqcup \bigsqcup_{u \in \tau} \{(\Sigma(u), N(u))\}                     \\
         & = P(\delta(\Sigma, v)).
    \end{align*}
    Arguing analogously, one can show that $P(\delta'(\Sigma', \Phi(v)^{-1})) = P(\delta(\Sigma, v^{-1}))$ when $N(v) \neq 2$. This shows that $\Phi$ satisfies (\ref{eq: power profile of transition}) for all $v \in \Lk(\sigma; \tau)$, which completes the proof.
\end{proof}

\begin{cor}
    \label{corollary: constant corollary}
    Let $(\Gamma, N)$ and $(\Gamma', N')$ be equivalent link-regular numbered graphs, and let $\mathcal{D}$ and $\mathcal{D}'$ be the FSAs constructed in Theorem \ref{theorem: fsa} that recognize the geodesic language of $\NGP(\Gamma)$ and $\NGP(\Gamma')$. Let $P$ and $Q$ be arbitrary power profiles of states in $\mathcal{D}$ and $\mathcal{D}'$. If $\Sigma$ and $\Sigma'$ are $P$-states of $\mathcal{D}$ and $\mathcal{D}'$, respectively, then the number of transitions from $\Sigma$ to a $Q$-state in $\mathcal{D}$ is the same as the number of transitions from $\Sigma'$ to a $Q$-state in $\mathcal{D}'$.
\end{cor}
\begin{proof}
    Let $\Phi : V\Gamma \to V\Gamma'$ be as in Lemma \ref{lemma: bijection between transitions}. By (\ref{eq: power profile of transition}), $\delta(\Sigma, v)$ is a $Q$-state if and only if $\delta(\Sigma', \Phi(v))$ is a $Q$-state. Because $\Phi$ is a bijection, there is a bijective correspondence between the transitions from $\Sigma$ to $Q$-states in $\mathcal{D}$ and the transitions from $\Sigma$ to $Q$-state in $\mathcal{D}'$, so we are done.
\end{proof}

Corollary \ref{corollary: constant corollary} allows us to define the following quantities $\beta_{P \to Q}$, which are well-defined up to equivalence of link-regular numbered graphs.

\begin{defi}
    Fix some numbered graph $(\Gamma, N)$, and let $\mathcal{D}$ be the FSA that recognizes the geodesic language of $\NGP(\Gamma)$. Given arbitrary power profiles $P$ and $Q$ of powered cliques corresponding to the states in $\mathcal{D}$, we define $\beta_{P \rightarrow Q}$ to be the number of transitions from a fixed $P$-state to a $Q$-state.
\end{defi}

\begin{thm}
    \label{theorem: main}
    If $(\Gamma, N)$ and $(\Gamma', N')$ are equivalent link-regular numbered graphs, then $\NGP(\Gamma)$ and $\NGP(\Gamma')$ have the same geodesic growth series.
\end{thm}
\begin{proof}
    Let $G = \NGP(\Gamma)$, and let $S$ be the standard generating set of $G$. We will show that the geodesic growth series of $(G, S)$ can be written in terms of the quantities $\beta_{P \rightarrow Q}$. Because the values $\beta_{P \rightarrow Q}$ are the same for equivalent link-regular numbered graphs, this proves the theorem.

    Let $d$ be the maximum clique size of $\Gamma$, and let $M$ be the maximum number $N(v)$ assigned to a vertex. We let $\mathcal{P}$ denote the collection of all multisets with elements in $\{-\lfloor M / 2 \rfloor, \dots, \lfloor M / 2 \rfloor\} \times \{N(v) \mid v \in V\}$ that have at most $d$ elements. Note that $\mathcal{P}$ contains all possible power profiles of states in $\mathcal{D}$. For each $P \in \mathcal{P}$, let $B_P(m)$ denote the collection of length $m$ words in $S^*$ that end at a $P$-state when processed by $\mathcal{D}$. Each geodesic of $(G, S)$ belongs to a unique set $B_P(m)$, so the total number of $m$-length geodesic words is given by the sum
    \[
        \sum_{P \in \mathcal{P}} |B_P(m)|.
    \]
    Therefore, the geodesic growth series of $(G,S)$ can be written as
    \[
        \sum_{m = 0}^\infty \left(\sum_{P \in \mathcal{P}} |B_P(m)|\right)z^m.
    \]

    Hence, it suffices to show that the quantities $|B_P(m)|$ can be written in terms of the quantities $\beta_{P \rightarrow Q}$. This is trivial for the case $m = 0$, where $|B_P(m)| = 0$ if $P$ is non-empty, and $|B_P(m)| = 1$ if $P$ is empty. Inductively, assume there is some $m > 0$ such that all the $|B_P(m)|$ can be written in terms of the $\beta_{P \rightarrow Q}$. We want to show that the same is true for all the $|B_P(m + 1)|$. Fix $P \in \mathcal{P}$. Observe that a word $x \in B_P(m)$ can be written as $x = ys$ for a word $y \in B_Q(m)$ with $Q \in \mathcal{P}$ and a single letter $s$. That is, a choice of a word $x \in B_P(m)$ consists of a choice of a power profile $Q \in \mathcal{P}$, then a choice of a word $y \in B_Q(m)$, and finally a choice of $s \in S$ for which there is an $s$-transition from the $Q$-state that $y$ ends at to some $P$-state. Hence, we have the recurrence
    \[
        |B_P(m)| = \sum_{Q \in \mathcal{P}} \beta_{Q \to P}|B_Q(m - 1)|.
    \]
    This proves that $|B_P(m)|$ can be written in terms of the quantities $\beta_{P \rightarrow Q}$, and so we are done by induction.
\end{proof}

Using Theorem \ref{theorem: main}, we are able to find examples of NGPs associated to non-isomorphic graphs that have the same geodesic growth series. For example, the NGPs associated to the numbered graphs from Example \ref{example: equivalence} have the same geodesic growth series, despite the graphs being non-isomorphic. More generally, the NGP associated to the disjoint union of identical $n$-cycles has the same geodesic growth series as the NGP associated to the $2n$-cycle with the same vertex numbers. Another example of this construction is the following:

\begin{exa}
    The disjoint union of the two $5$-cycles and the $10$-cycle shown below are equivalent as link-regular numbered graphs, so their associated NGPs have the same geodesic growth.
    \begin{center}
        \begin{tikzpicture}
            \foreach \th/\lab in {1/2, 2/3, 3/4, 4/5, 5/6} {
                    \node[label={[label distance=-0.1cm]\th * 360 / 5 + 360 / 20:$\lab$}] (v\th) at (\th * 360 / 5 + 360 / 20:1) {};
                    \filldraw (v\th) circle (2pt);
                }
            \draw (v1.center) -- (v2.center) -- (v3.center) -- (v4.center) -- (v5.center) -- cycle;

            \foreach \th/\lab in {1/2, 2/3, 3/4, 4/5, 5/6} {
                    \node[label={[label distance=-0.1cm]\th * 360 / 5 + 360 / 20:$\lab$}] (v\th) at ($(3, 0) + (\th * 360 / 5 + 360 / 20:1)$) {};
                    \filldraw (v\th) circle (2pt);
                }
            \draw (v1.center) -- (v2.center) -- (v3.center) -- (v4.center) -- (v5.center) -- cycle;

            \foreach \th/\lab in {1/2, 2/3, 3/4, 4/5, 5/6, 6/2, 7/3, 8/4, 9/5, 10/6} {
                    \node[label={[label distance=-0.1cm]\th * 360 / 10:$\lab$}] (v\th) at ($(7.5, 0) + (\th * 360 / 10:1.5)$) {};
                    \filldraw (v\th) circle (2pt);
                }
            \draw (v1.center) -- (v2.center) -- (v3.center) -- (v4.center) -- (v5.center) -- (v6.center) -- (v7.center) -- (v8.center) -- (v9.center) -- (v10.center) -- cycle;
        \end{tikzpicture}
    \end{center}
\end{exa}

These examples demonstrate the following corollary to Theorem \ref{theorem: main}.

\begin{cor}
    There are infinitely many pairs of NGPs associated to non-isomorphic graphs that have the same geodesic growth.\qed
\end{cor}

\section{Right-angled coxeter groups}
\label{sec: theorem 2}
We now shift our attention to a specific class of NGPs known as \textit{right-angled Coxeter groups}.

\begin{defi}[Right-angled Coxeter groups]
    A \textit{right-angled Coxeter group} (RACG) is an NGP associated to a numbered graph whose vertex numbers are all $2$. In particular, we may speak of the \textit{RACG associated to a simplicial graph $\Gamma$}, written $\RACG(\Gamma)$, which we define to be the NGP associated to the numbered graph $(\Gamma, N)$ where $N(v) = 2$ for all $v \in V(\Gamma)$.

    As with NGPs, whenever we refer to the geodesics (or geodesic growth) of some RACG, we are working with respect to the standard generating set.
\end{defi}

The results we have so far derived for NGPs become much simpler in the context of RACGs. For instance, consider the FSA $\mathcal{D}$ we constructed in Theorem \ref{theorem: fsa} that recognizes the geodesic language of some RACG. The states of $\mathcal{D}$ are powered cliques, but because every vertex number is $2$, a powered clique can only assign the powers $0$ and $1$ to vertices. Thus, we may identify a powered clique with its support, which must be a clique. In particular, the states of $\mathcal{D}$ are in bijection with the cliques of the graph. Rewriting the transition map in terms of cliques, we obtain the following special case of Theorem \ref{theorem: fsa}.

\begin{prop}
    \label{proposition: racg fsa}
    The geodesic language of $\RACG(\Gamma)$ is accepted by the FSA $\mathcal{D}$, where the states are the cliques of $\Gamma$ plus a designated reject state $q_\mathrm{rej}$, the start state is the empty set, and the transition map $\delta$ is defined by
    \[
        \delta(\sigma, v) =
        \begin{cases}
            (\Lk(v) \cap \sigma) \cup \{v\} & \text{if } v \notin \sigma, \\
            q_\mathrm{rej}                  & \text{otherwise}.
        \end{cases}
    \]
\end{prop}

Similarly, the definition of link-regularity given in Section \ref{sec: theorem 1} can be simplified when dealing with RACGs. Consider a numbered graph $(\Gamma, N)$ whose vertex numbers are all $2$. Recall that $(\Gamma, N)$ is link-regular if $N(\sigma_1) = N(\sigma_2)$ implies $N(\Lk(\sigma_1)) = N(\Lk(\sigma_2))$ for cliques $\sigma_1, \sigma_2$ of $\Gamma$. But since every vertex number is $2$, the condition $N(U_1) = N(U_2)$ for subsets $U_1, U_2 \subseteq V(\Gamma)$ holds precisely when $|U_1| = |U_2|$. Therefore, $(\Gamma, N)$ is link-regular if and only if $|\sigma_1| = |\sigma_2|$ implies $\abs{\Lk(\sigma_1)} = \abs{\Lk(\sigma_2)}$ for all cliques $\sigma_1, \sigma_2$. This condition motivates the following definition:

\begin{defi}[Link-regularity of simplicial graphs]
    \label{def: link regularity of simplicial graphs}
    Let $\Gamma$ be a simplicial graph. We say that $\Gamma$ is \textit{link-regular} if whenever two cliques $\sigma_1, \sigma_2$ of $\Gamma$ have the same size, their links $\Lk(\sigma_1), \Lk(\sigma_2)$ have the same size.
\end{defi}

As before, we can define two link-regular simplicial graphs to be equivalent if their disjoint union is link-regular. In this setting, Theorem \ref{theorem: main} says that RACGs associated to equivalent link-regular simplicial graphs have identical geodesic growth series. This is one of the main results of \cite{Ant12}. In this section, we generalize their result by deriving the following closed-form expression for geodesic growth.

\begin{thm}
    \label{thm: main RACG theorem}
    Let $\Gamma$ be a link-regular simplicial graph with maximum clique size $d$. Let $\ell_0 = |V(\Gamma)|$, and for $1 \leq k \leq d$, let $\ell_k = \abs{\Lk(\sigma)}$, where $\sigma$ is any $k$-clique (this is well-defined by link-regularity of $\Gamma$). For integers $0 \leq m \leq d$ and $0 \leq k \leq m$, set
    \[
        N_{m,k} =
        \begin{dcases}
            \left(\prod_{k < j < m} \ell_j\right)\sum_{j = k}^m \binom{m - k}{j - k} (-1)^{j - k} \ell_j
                        & \text{if $k < m - 1$} \\
            \ell_{m - 1} - \ell_m - 1 & \text{if $k = m - 1$} \\
            1 & \text{if $k = m$},
        \end{dcases}
    \]
    and for integers $0 \leq i \leq d$ and $0 \leq j \leq i$, set
    \[
        \begin{dcases}
            M_{i, j} = \sum_{\substack{j \leq s_1 < t_1 \leq \dots \leq s_n < t_n \leq d \\
                (t_1 - s_1) + \dots + (t_n - s_n) = i - j}}
        (-1)^n \binom{t_1}{s_1} \cdots \binom{t_n}{s_n} N_{t_1, s_1} \cdots N_{t_n, s_n}
        & \text{if $i \neq j$} \\
            M_{i, j} = 1 & \text{if $i = j$.} \\
        \end{dcases}
    \]
    Then the geodesic growth of $\RACG(\Gamma)$ is given by the rational function
    \[
        \mathcal{G}(z) =
        \frac{\sum_{i = 0}^d \left(\sum_{j = 0}^i \ell_0 \cdots \ell_{j - 1} M_{i, j}\right)z^i}
        {\sum_{i = 0}^d M_{i, 0} z^i}.
    \]
\end{thm}

Previously, the geodesic growth series of $\RACG(\Gamma)$ was only known for maximum clique size $d \leq 3$, proven in \cite{Ant12} and \cite{Ant13}. Detailed calculations using Theorem \ref{thm: main RACG theorem}, and a Sage program that automatically carries out these calculations, can be found in \cite{Mar23}. This program was used to compute the following formula when the maximum clique size is $d = 4$.

\begin{cor}
    If $\Gamma$ is a link-regular simplicial graph that does not contain $5$-cliques, the geodesic growth of $\RACG(\Gamma)$ is the rational function $q(z) / p(z)$, where
    \begin{align*}
        q(z) ={}
         & 24z^4 + (-\ell_1\ell_2\ell_3 + 4\ell_1\ell_2 + \ell_2\ell_3 - 12\ell_1 - 4\ell_2 - 2\ell_3 + 50)z^3
        \\
         & + (\ell_1\ell_2 + \ell_1\ell_3 + \ell_2\ell_3 - 7\ell_1 - 5\ell_2 - 3\ell_3 + 35)z^2 + (-\ell_1 - \ell_2 - \ell_3 + 10)z + 1,
        \\
        p(z) ={}
         & (\ell_0\ell_1\ell_2\ell_3 - 4\ell_0\ell_1\ell_2 + 12\ell_0\ell_1 - 24\ell_0 + 24)z^4
        \\
         & + (-\ell_0\ell_1\ell_2 - \ell_0\ell_1\ell_3 - \ell_0\ell_2\ell_3 - \ell_1\ell_2\ell_3 + 7\ell_0\ell_1 + 4\ell_0\ell_2 + 4\ell_1\ell_2 + 2\ell_0\ell_3
        \\
         & \phantom{+(} \ + \ell_2\ell_3 - 26\ell_0 - 12\ell_1 - 4\ell_2 - 2\ell_3 - 50)z^3
        \\
         & + (\ell_0\ell_1 + \ell_0\ell_2 + \ell_1\ell_2 + \ell_0\ell_3 + \ell_1\ell_3 + \ell_2\ell_3 - 9\ell_0 - 7\ell_1 - 5\ell_2 - 3\ell_3 + 35)z^2
        \\
         & + (-\ell_0 - \ell_1 - \ell_2 - \ell_3 + 10)z + 1. \tag*{\qed}
    \end{align*}
\end{cor}

The remainder of this section is devoted to the proof of Theorem \ref{thm: main RACG theorem}. We fix the following notation. Let $\Gamma$, $d$, and $\ell_k$ be defined as in the statement of Theorem \ref{thm: main RACG theorem}. Let $S$ be the standard generating set of $\RACG(\Gamma)$. Let $\mathcal{C}$ denote the set of all cliques of $\Gamma$. Let $\widetilde{\mathcal{C}}_m$ denote the set of all \textit{ordered $m$-cliques} of $\Gamma$, i.e., tuples $(v_1, \dots, v_m)$ consisting of distinct vertices such that $\{v_1, \dots, v_m\}$ is a clique. The elements of $\widetilde{\mathcal{C}}_m$ will be denoted $\tilde{\sigma}$, where $\sigma$ is the underlying unordered clique.

We begin by using Theorem \ref{theorem: chomsky black magic} to obtain a system of equations that allows us to solve for the geodesic growth of $\RACG(\Gamma)$. Applying Proposition \ref{proposition: grammar rules} to the FSA $\mathcal{D}$ in Proposition \ref{proposition: racg fsa}, we obtain the following grammar that accepts the geodesic language of $\RACG(\Gamma)$: to each clique $\sigma \in \mathcal{C}$, we associate the grammar variable $\mathbf{\Delta}_\sigma$, and the production rules
\begin{align*}
    \mathbf{\Delta}_\sigma & \longrightarrow v \mathbf{\Delta}_{(\Lk(v) \cap \sigma) \cup \{v\}}
    \quad \text{(for $v \not\in \sigma$)}                                                       \\
    \mathbf{\Delta}_\sigma & \longrightarrow \varepsilon
\end{align*}
The start symbol is the variable $\mathbf{\Delta}_\varnothing$. Observe that the vertices not in $\sigma$ can be counted by first choosing a (possibly empty) subset $\tau \subseteq \sigma$, and then choosing a vertex in $\Lk(\sigma; \tau)$; more precisely, we have
\[
    V(\Gamma) \setminus \sigma
    = \bigsqcup_{\tau \subseteq \sigma} \Lk(\sigma; \tau).
\]
Hence, the above grammar rules can be equivalently written
\begin{equation}
    \label{eq: grammar}
    \begin{aligned}
        \mathbf{\Delta}_\sigma & \longrightarrow u \mathbf{\Delta}_{\tau \cup \{u\}}
        \quad \text{(for $\tau \subseteq \sigma$, $u \in \Lk(\sigma; \tau)$)}       \\
        \mathbf{\Delta}_\sigma & \longrightarrow \varepsilon
    \end{aligned}
\end{equation}
To be able to apply Theorem \ref{theorem: chomsky black magic}, we first need to check two conditions:

\begin{lem}
    \label{lemma: unambiguous}
    The grammar in (\ref{eq: grammar}) is unambiguous, and every variable is reachable.
\end{lem}
\begin{proof}
    We first show that (\ref{eq: grammar}) is unambiguous. Observe that, given any grammar variable $\mathbf{\Delta}_\sigma$ and some vertex $u \in V\Gamma$, there is at most one rule $\mathbf{\Delta}_\sigma \to \alpha$ where the right-hand side $\alpha$ contains $u$. This implies that, if the grammar generates some word $x = v_1 \cdots v_n \in S^*$, then there is only one possible production rule that could have been applied at each step. Hence, $x$ has a unique derivation from the start symbol, as required.
    
    For the second claim, fix a variable $\mathbf{\Delta}_\sigma$ for some clique $\sigma = \{v_1, \dots, v_m\}$. Then we have the derivation
    \[
        \mathbf{\Delta}_\varnothing
        \Rightarrow v_1\mathbf{\Delta}_{\{v_1\}}
        \Rightarrow v_1v_2\mathbf{\Delta}_{\{v_1, v_2\}}
        \cdots
        \Rightarrow v_1 \cdots v_n\mathbf{\Delta}_{\{v_1, \dots, v_n\}},
    \]
    which proves that $\mathbf{\Delta}_\sigma$ is reachable.
\end{proof}

Now, we may apply Theorem \ref{theorem: chomsky black magic} to obtain the equations
\begin{equation}
    \label{eq: bad system}
    \Delta_{\sigma}(z) = 1 + z\sum_{\tau \subseteq {\sigma}} \sum_{u \in \Lk(\sigma; \tau)}\Delta_{\tau \cup \{u\}}(z)
\end{equation}
for all $\sigma \in \mathcal{C}$, where $\Delta_\sigma(z)$ denotes the growth series of the language consisting of all words that can be derived from $\mathbf{\Delta}_\sigma$. To simplify this system, we introduce the power series
\[
    \mathcal{G}_m(z) = \sum_{\tilde{\sigma} \in \widetilde{\mathcal{C}}_m} \Delta_\sigma(z).
\]
We want to convert the system of equations (\ref{eq: bad system}), which is written in terms of the power series $\Delta_\sigma(z)$, into a system of equations written in terms of the power series $\mathcal{G}_m(z)$. The following proposition is the first step in doing so.

\begin{prop}
    We have
    \[
        \mathcal{G}_0(z) = 1 + z\mathcal{G}_1(z)
    \]
    and
    \[
        \mathcal{G}_m(z)
        = \ell_0\dotsb\ell_{m-1} + \sum_{k=0}^m \binom{m}{k} \sum_{(v_1,\dotsc,v_m) \in \widetilde{\mathcal{C}}_m} \sum_{u \in \Lk(\{v_1,\dotsc,v_m\}; \{v_1,\dotsc,v_k\})} \Delta_{\{v_1,\dotsc,v_k,u\}}(z).
    \]
    The geodesic growth series of $\RACG(\Gamma)$ is then given by $\mathcal{G}_0(z)$.
\end{prop}
\begin{proof}
    First, we derive the equation for $\mathcal{G}_0(z)$. Because the only $0$-clique is the empty clique, we have $\mathcal{G}_0(z) = \Delta_\varnothing(z)$. In particular, since $\mathbf{\Delta}_\varnothing$ is the start symbol, $\mathcal{G}_0(z) = \Delta_\varnothing(z)$ is the geodesic growth series of $\RACG(\Gamma)$. By (\ref{eq: bad system}), we get
    \begin{align*}
        \mathcal{G}_0(z)
         & = 1 + z\sum_{u \in \Lk(\varnothing; \varnothing)} \Delta_{\{u\}}(z) \\
         & = 1 + z\sum_{u \in V(\Gamma)} \Delta_{\{u\}}(z)                     \\
         & = 1 + z\mathcal{G}_1(z),
    \end{align*}
    as required.

    Next, we derive the equation for $\mathcal{G}_m(z)$ when $1 \leq m \leq d$. Expanding the summation in (\ref{eq: bad system}), we get
    \[
        \Delta_{\sigma}(z) = 1 + z\sum_{k = 0}^m\sum_{\substack{\tau \subseteq {\sigma} \\ |\tau| = k}} \sum_{u \in \Lk(\sigma; \tau)}\Delta_{\tau \cup \{u\}}(z)
    \]
    for all $\sigma \in \mathcal{C}$. Summing over all ordered cliques in $\widetilde{\mathcal{C}}_m$, we get
    \[
        \mathcal{G}_m(z) = |\widetilde{\mathcal{C}}_m|
        + z\sum_{k = 0}^m\sum_{\tilde{\sigma} \in \widetilde{\mathcal{C}}_m}\sum_{\substack{\tau \subseteq {\sigma} \\ |\tau| = k}} \sum_{u \in \Lk(\sigma; \tau)}
        \Delta_{\tau \cup \{u\}}(z).
    \]
    Now, we determine the size of $\widetilde{\mathcal{C}}_m$. Since an ordered $1$-clique is just a vertex, we have $|\widetilde{\mathcal{C}}_1| = |V(\Gamma)| = \ell_0$. For $m > 0$, every ordered clique of size $m$ is obtained by taking an ordered $(m - 1)$-clique $\tilde{\sigma}$ and appending a vertex $v \in \Lk(\sigma)$. There are $\ell_{m - 1} = |\Lk(\sigma)|$ choices for $v$, and so $|\widetilde{\mathcal{C}}_m| = \ell_{m - 1}|\widetilde{\mathcal{C}}_m|$. By induction, we get $|\widetilde{\mathcal{C}}_m| = \ell_0 \cdots \ell_{m - 1}$. Thus, we obtain the equation
    \begin{equation}
        \label{eq:sumaftercountingorderedcliques}
        \mathcal{G}_m(z)
        = \ell_0\dotsb\ell_{m-1} + z\sum_{k = 0}^m \sum_{\tilde{\sigma} \in \widetilde{\mathcal{C}}_m}\sum_{\substack{\tau \subseteq \sigma \\ |\tau| = k}} \sum_{u \in \Lk(\sigma; \tau)} \Delta_{\tau \cup \{u\}}(z).
    \end{equation}

    Next, we want to interchange the second and third sums in (\ref{eq:sumaftercountingorderedcliques}). However, the third sum currently depends on the second sum, so we must first remove this dependency. Let $[m] = \{1, \dots, m\}$, and fix an arbitrary ordered $m$-clique $\tilde{\sigma} = (v_1,\dotsc,v_m)$. Given a subset $T = \{i_1,\dotsc,i_k\} \subseteq [m]$ of size $k$, define $\pi_T(\tilde{\sigma})$ to be the subclique $\{v_{i_1},\dotsc,v_{i_k}\}$ of size $k$. Observe that $\pi_T$ defines a bijection between the $k$-element subsets of $[m]$ and the subcliques of $\sigma$ of size $k$. Thus, changing variables in (\ref{eq:sumaftercountingorderedcliques}), we obtain the equation
    \[
        \mathcal{G}_m(z) = \ell_0\dotsb\ell_{m-1} + \sum_{k=0}^m \sum_{\tilde{\sigma} \in \widetilde{\mathcal{C}}_m} \sum_{\substack{T \subseteq [m]\\ |T| = k}} \sum_{u \in \Lk(\sigma; \pi_T(\tilde{\sigma}))} \Delta_{\pi_T(\tilde{\sigma}) \cup \{u\}}(z).
    \]
    Now, the third sum no longer depends on the second sum, so we can write
    \begin{equation}
        \label{eq:sumafterswap}
        \mathcal{G}_m(z) = \ell_0\dotsb\ell_{m-1} + \sum_{k=0}^m \sum_{\substack{T \subseteq [m]\\ |T| = k}} \sum_{\tilde{\sigma} \in \widetilde{\mathcal{C}}_m} \sum_{u \in \Lk(\sigma; \pi_T(\tilde{\sigma}))} \Delta_{\pi_T(\tilde{\sigma}) \cup \{u\}}(z).
    \end{equation}

    Finally, we show that the second sum in (\ref{eq:sumafterswap}) can be collapsed. Fix $0 \leq k \leq m$. We claim that if $T,T' \subseteq [1\dotsc m]$ with $|T| = |T'| = k$, then
    \begin{equation}
        \label{eq:claim for collapsing sum}
        \sum_{\tilde{\sigma} \in \widetilde{\mathcal{C}}_m} \sum_{u \in \Lk(\sigma; \pi_T(\tilde{\sigma}))} \Delta_{\pi_{T}(\tilde{\sigma}) \cup \{u\}}(z) = \sum_{\tilde{\sigma} \in \widetilde{\mathcal{C}}_m} \sum_{u \in \Lk(\sigma; \pi_{T'}(\tilde{\sigma}))} \Delta_{\pi_{T'}(\tilde{\sigma}) \cup \{u\}}(z).
    \end{equation}
    Write $T = \{i_1,\dotsc,i_k\}$ and $T'= \{j_1,\dotsc,j_k\}$. There is a permutation $p: \{1,\dotsc,m\} \rightarrow \{1,\dotsc,m\}$ such that $i_\ell \mapsto j_\ell$ for $1 \leq \ell \leq k$. Then $p$ induces a bijection $\tilde{p} : \widetilde{\mathcal{C}}_m \rightarrow \widetilde{\mathcal{C}}_m$ by
    \[
        \tilde{\sigma} = (v_1,\dotsc,v_m) \mapsto (v_{p(1)},\dotsc,v_{p(m)}) =: \tilde{p}(\tilde{\sigma}).
    \]
    Observe that the map $\tilde{p}$ satisfies $\pi_T(\tilde{p}(\tilde{\sigma})) = \{v_{j_1},\dotsc,v_{j_k}\} = \pi_{T'}(\tilde{\sigma})$. Also, the underlying clique of $\tilde{p}(\tilde{\sigma})$ is just $\sigma$ itself. Therefore, equation (\ref{eq:claim for collapsing sum}) follows from applying the change of variables $\tilde{p}$ to the outer sum. Applying this result with $T' = \{1, \dots, k\}$, we get
    \begin{align*}
         & \sum_{\substack{T \subseteq [m]
        \\ |T| = k}} \sum_{\tilde{\sigma} \in \widetilde{\mathcal{C}}_m} \sum_{u \in \Lk(\sigma; \pi_T(\tilde{\sigma}))} \Delta_{\pi_T(\tilde{\sigma}) \cup \{u\}}(z) \\
         & \qquad = \binom{m}{k} \sum_{(v_1,\dotsc,v_m) \in \widetilde{\mathcal{C}}_m} \sum_{u \in \Lk(\{v_1,\dotsc,v_m\}; \{v_1,\dotsc,v_k\})} \Delta_{\{v_1,\dotsc,v_k,u\}}(z).
    \end{align*}
    Substituting this into (\ref{eq:sumafterswap}), we obtain the required formula for $\mathcal{G}_m(z)$.
\end{proof}

Given an ordered clique $(v_1,\dotsc,v_k,u) \in \mathcal{C}_{k+1}$, we define the power series $\Delta_{(v_1,\dotsc,v_k,u)}(z)$ to equal the power series $\Delta_{\{v_1,\dotsc,v_k,u\}}(z)$. We do this to improve notation in order to keep track of the ordering for Proposition \ref{proposition: supercounting}. Making use of this change converts our system into the equivalent form
\begin{equation}
    \label{eq: precountingsystem}
    \mathcal{G}_m(z) = \ell_0\dotsb\ell_{m-1} + \sum_{k=0}^m \binom{m}{k} \sum_{(v_1,\dotsc,v_m) \in \widetilde{\mathcal{C}}_m} \sum_{u \in \Lk(\{v_1,\dotsc,v_m\}; \{v_1,\dotsc,v_k\})} \Delta_{(v_1,\dotsc,v_k,u)}(z)
\end{equation}
for $1 \leq m \leq d$.

\begin{prop}
    \label{proposition: supercounting}
    For $1 \leq m \leq d$, $0 \leq k \leq m$, and an arbitrary ordered clique $(v_1,\dotsc,v_k,u)$, the number of times $\Delta_{(v_1,\dotsc,v_k,u)}(z)$ appears in the sum
    \begin{equation}
        \label{eq: sumtocount}
        \sum_{(v_1,\dotsc,v_m) \in \widetilde{\mathcal{C}}_m} \sum_{u \in \Lk(\{v_1,\dotsc,v_m\}; \{v_1,\dotsc,v_k\})} \Delta_{(v_1,\dotsc,v_k,u)}(z)
    \end{equation}
    is an integer $N_{m, k}$ that is independent of the vertices $v_1, \dots, v_k, u$. Furthermore, the integers $N_{m, k}$ satisfy the following properties:
    \begin{itemize}
        \item $N_{m,m} = 1$.
        \item $N_{m, m-1} = \ell_{m-1} - \ell_m - 1$.
        \item If $k < m-1$, then $N_{m, k}$ satisfies the recurrence
        \[
            N_{m,k} = \ell_{m-1}N_{m-1,k} - \ell_{k+1}N_{m,k+1}.
        \]
    \end{itemize}
\end{prop}

\begin{proof}
    Fix indices $m, k$ and an ordered clique $(v_1, \dots, v_k, u)$. Let $\tilde{\tau} = (v_1,\dotsc,v_k)$. The term $\Delta_{(v_1,\dotsc,v_k,u)}(z)$ appears once in the sum (\ref{eq: sumtocount}) for each ordered clique $\tilde{\sigma} \in \widetilde{\mathcal{C}}_m$ where $\sigma \supseteq \tau$ and $u \in \Lk(\sigma; \tau)$. Thus, we want to count the number of possible ordered cliques $\tilde{\sigma}$.

    Each such $\tilde{\sigma}$ is given by choosing $m - k$ vertices $v_{k+1},\dotsc,v_m \in \Lk(\{v_1,\dotsc,v_k\})$ so that $(v_1,\dotsc,v_m)$ forms an ordered clique, with the additional restriction that $v_{k+1},\dotsc,v_{m} \notin \Lk(u)$ (cf. Figure \ref{fig:countingcliques1}).
    \begin{figure}
        \centering
        \includegraphics[scale=1]{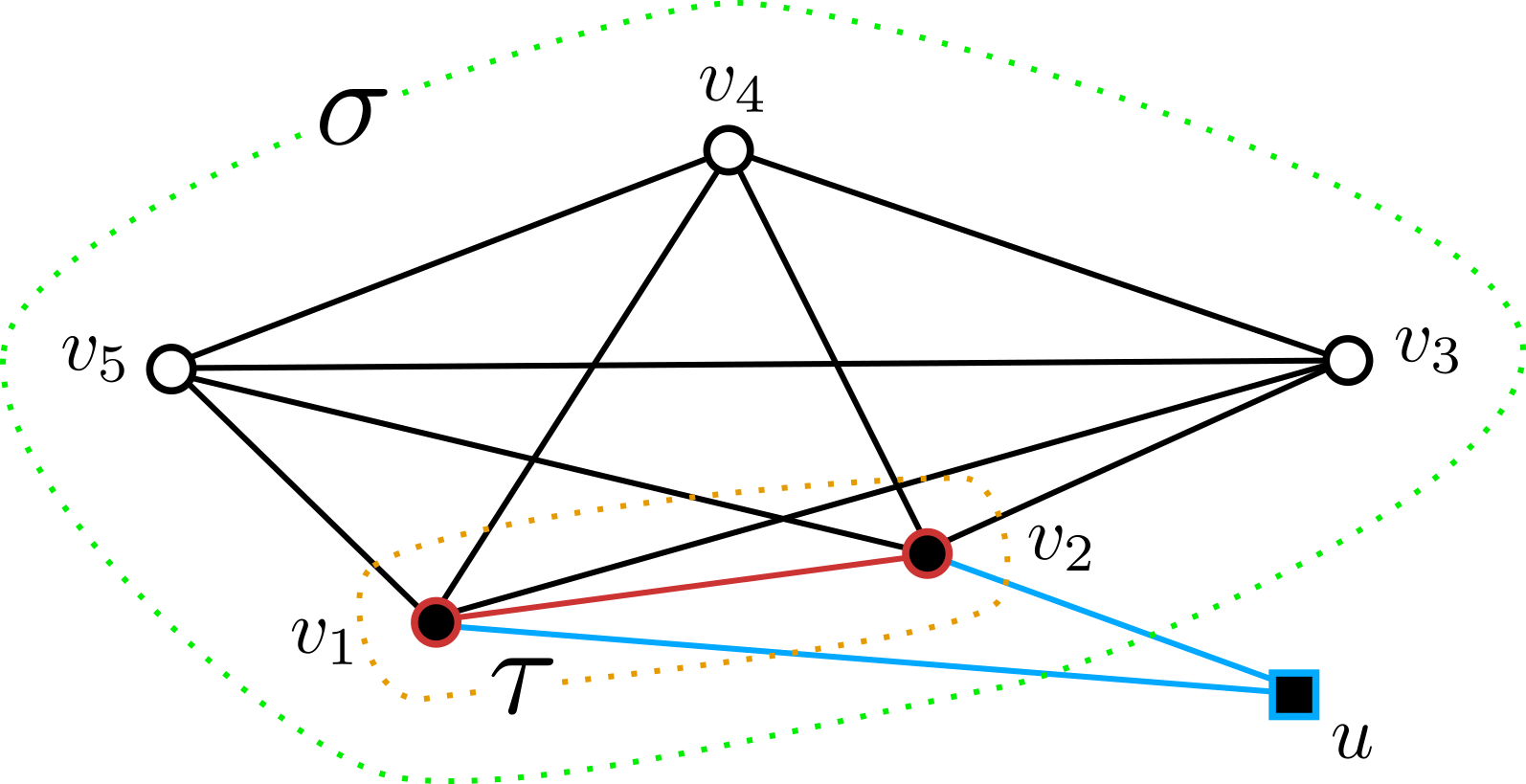}
        \caption{A picture showcasing $\tilde{\tau} = (v_1,\dotsc,v_k)$, $\tilde{\sigma} = (v_1,\dotsc,v_m)$, and $u$. The solid black vertices are fixed, and the open circles are what is free to be chosen. The vertices $\{v_1,\dotsc,v_m\}$ form a clique and $u$ is not connected to $v_{k+1},\dotsc,v_m$ (here, $k = 2$ and $m = 5$).}
        \label{fig:countingcliques1}
    \end{figure}
    Let $N_{m,k}(\tilde{\tau}, u)$ be the number of ways to choose $m-k$ vertices in this way. We distinguish three cases.
    
    \medskip
    
    \noindent\textit{Case 1 ($k = m$).} Here, $N_{m,m}(\tilde{\tau}, u)$ is the number of ways to choose $0$ vertices such that $(v_1, \dots, v_m)$ is an ordered clique. Hence, $N_{m, m}(\tilde{\tau}, u) = 1$ and is independent of $\tilde{\tau}$ and $u$, so we may write $N_{m, m} = N_{m, m}(\tilde{\tau}, u)$.
    
    \medskip
    
    \noindent\textit{Case 2 ($k = m - 1$).} There are $\ell_{m - 1} - 1$ ways to choose $v_m \in \Lk(\{v_1, \dots, v_{m - 1}\})$ with $v_m \neq u$ to obtain an ordered clique $\tilde{\sigma} = (v_1, \dots, v_m)$. We want $v_m \not\in \Lk(u)$, so we subtract off the $\ell_m$ ways to choose $v_m \in \Lk(\{v_1, \dots, v_{m - 1}, u\})$. This shows $N_{m, m - 1}(\tilde{\tau}, u) = \ell_{m - 1} - \ell_m - 1$ and is independent of $\tilde{\tau}$ and $u$, so we may write $N_{m, m - 1} = N_{m, m - 1}(\tilde{\tau}, u)$.
    
   \begin{figure}
        \centering
        \includegraphics[scale=0.9]{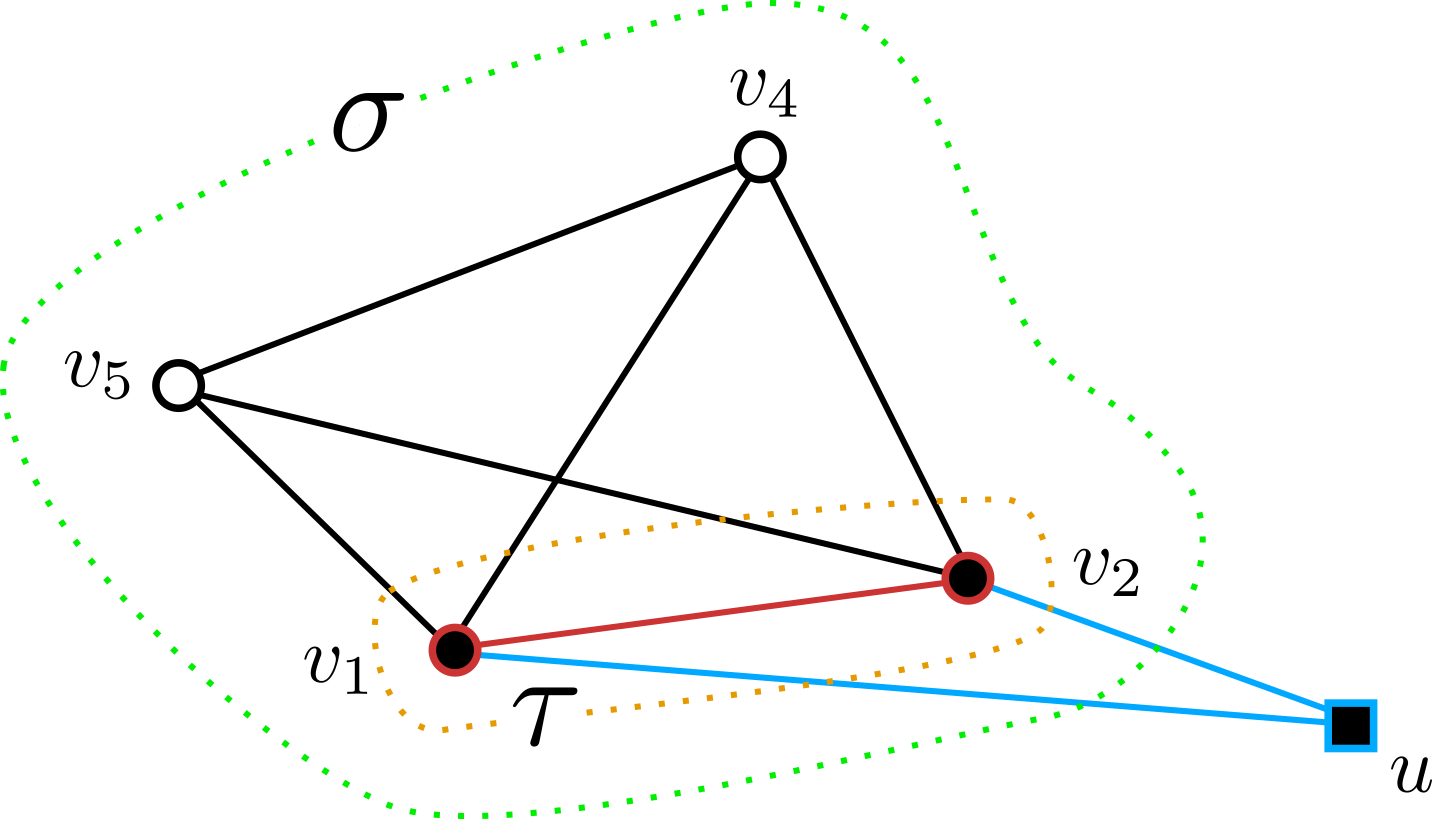}
        \includegraphics[scale=0.9]{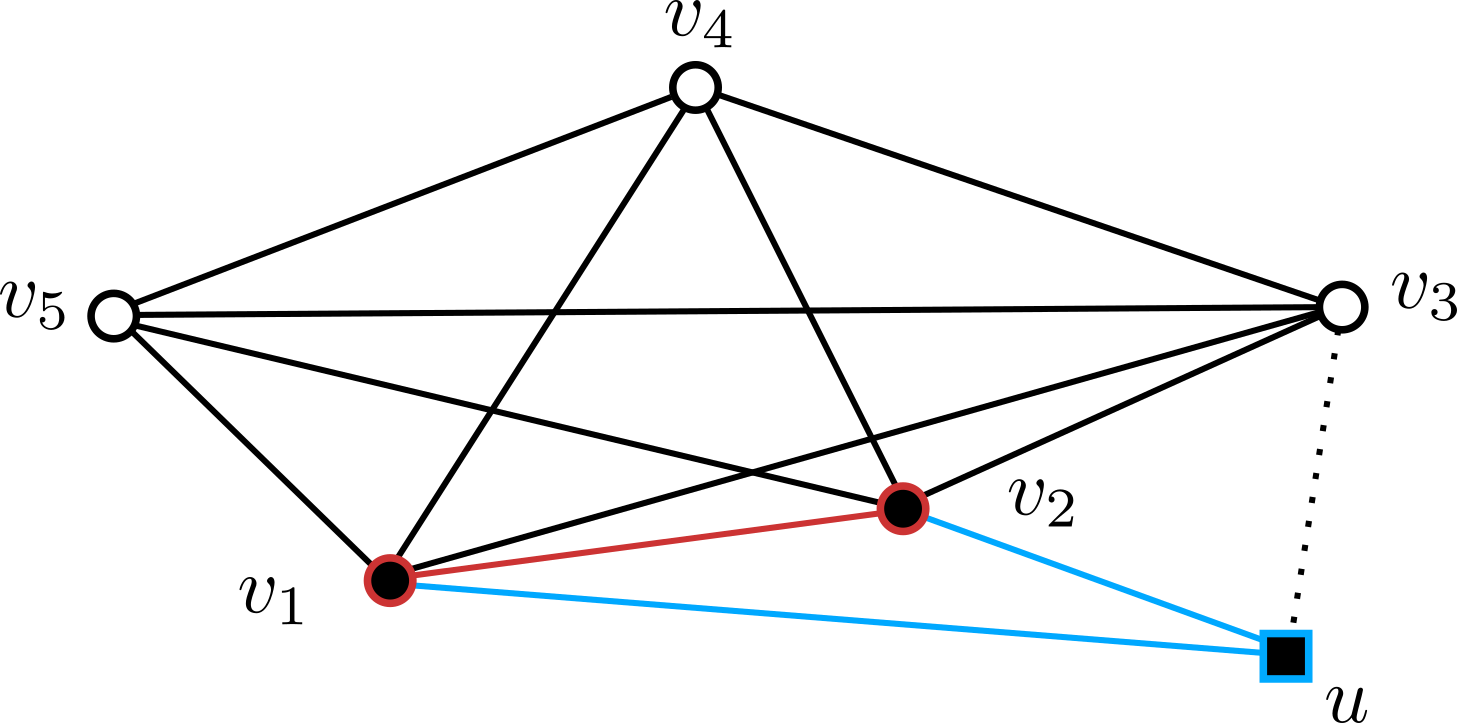}
        \caption{The number of ways to choose $v_{k+2},\dotsc,v_m$ to obtain the figure on the left is given by $N_{m-1,k}(\tilde{\tau}, u)$. Choosing $v_{k+1} \in \Lk(\{v_1,\dotsc,v_{k}, v_{k+2}, \dotsc,v_m\})$ gives the figure on the right, and there are $\ell_{m-1}$ ways to choose such a $v_k$. The dotted line means $v_k$ may or may not share an edge with $u$.}
        \label{fig:countingcliques2and3}
    \end{figure}
    
    \medskip
    
    \noindent\textit{Case 3 ($k < m - 1$).} First, we count the number of ways to choose $v_{k + 1}, \dots, v_m$ such that $v_{k + 2}, \dots, v_m \not\in \Lk(u)$ and $(v_1, \dots, v_m)$ forms an ordered clique (in particular, we are allowing $v_{k + 1} \in \Lk(u)$). This can be done by counting the number of ways to complete the following two steps (cf. Figure \ref{fig:countingcliques2and3}):
    \begin{itemize}
        \item[($i$)] Choose vertices $v_{k+2},\dotsc,v_m \notin \Lk(\{u\})$ such that $(v_1,\dotsc, v_{k}, v_{k + 2}, \dotsc,v_m)$ forms an ordered clique.
        \item[($ii$)] Choose a vertex $v_{k + 1} \in \Lk(\{v_1, \dotsc, v_{k}, v_{k + 2}, \dotsc,v_m\})$ to obtain an ordered clique $\tilde{\sigma} = (v_1,\dotsc,v_m)$.
    \end{itemize}
    By definition, there are $N_{m-1,k}(\tilde{\tau}, u)$ ways to complete ($i$) and $\ell_{m-1}$ ways to complete ($ii$), so there are $\ell_{m-1}N_{m-1,k}(\tilde{\tau}, u)$ ways to complete both ($i$) and ($ii$). 
    
    We have overcounted by precisely the cases where $v_{k + 1} \in \Lk(u)$. Thus, we must subtract off the number of ways to choose $v_{k + 1}, \dots, v_m$ such that $v_{k + 1} \in \Lk(u)$, $v_{k + 2}, \dots, v_m \not\in \Lk(u)$, and $(v_1, \dots, v_m)$ forms an ordered clique. This can be done by counting the number of ways of to complete the following two steps (cf. Figure \ref{fig:countingcliques4and5}):
    \begin{itemize}
        \item[($iii$)] Choose $v_{k+1} \in \Lk(\{v_1,\dotsc,v_k,u\})$, so that $(v_1, \dots, v_{k + 1})$ is an ordered clique and $v_{k + 1} \in \Lk(u)$.
        \item[($iv$)] Choose vertices $v_{k+2},\dotsc,v_m \notin \Lk(u)$ such that $(v_1,\dotsc,v_m)$ forms an ordered clique.
    \end{itemize}
    There are $\ell_{k + 1}$ ways to complete ($iii$) and $N_{m, k + 1}((v_1, \dots, v_{k + 1}), u)$ to complete ($iv$), so there are $\ell_{k + 1}N_{m, k + 1}((v_1, \dots, v_{k + 1}), u)$ ways to complete both ($iii$) and ($iv$).
    
    From this, we obtain the recurrence relation
    \[
        N_{m, k}(\tilde{\tau}, u)
        = \ell_{m - 1}N_{m - 1, k}(\tilde{\tau}, u)
        - \ell_{k + 1}N_{m, k + 1}((v_1, \dots, v_{k + 1}), u).
    \]
    In Case 1 and Case 2, we have shown that the base cases of this recurrence are independent of $\tilde{\tau}$ and $u$. Therefore, $N_{m, k}(\tilde{\tau}, u)$ is independent of $\tilde{\tau}$ and $u$. Writing $N_{m, k} = N_{m, k}(\tilde{\tau}, u)$, we obtain the required formula.
    \begin{figure}
        \centering
        \includegraphics[scale=0.9]{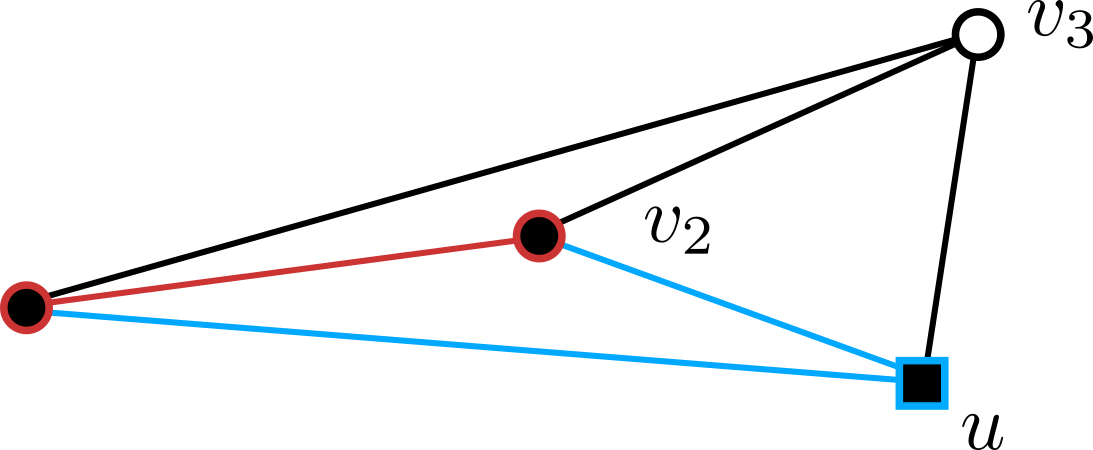}
        \includegraphics[scale=0.9]{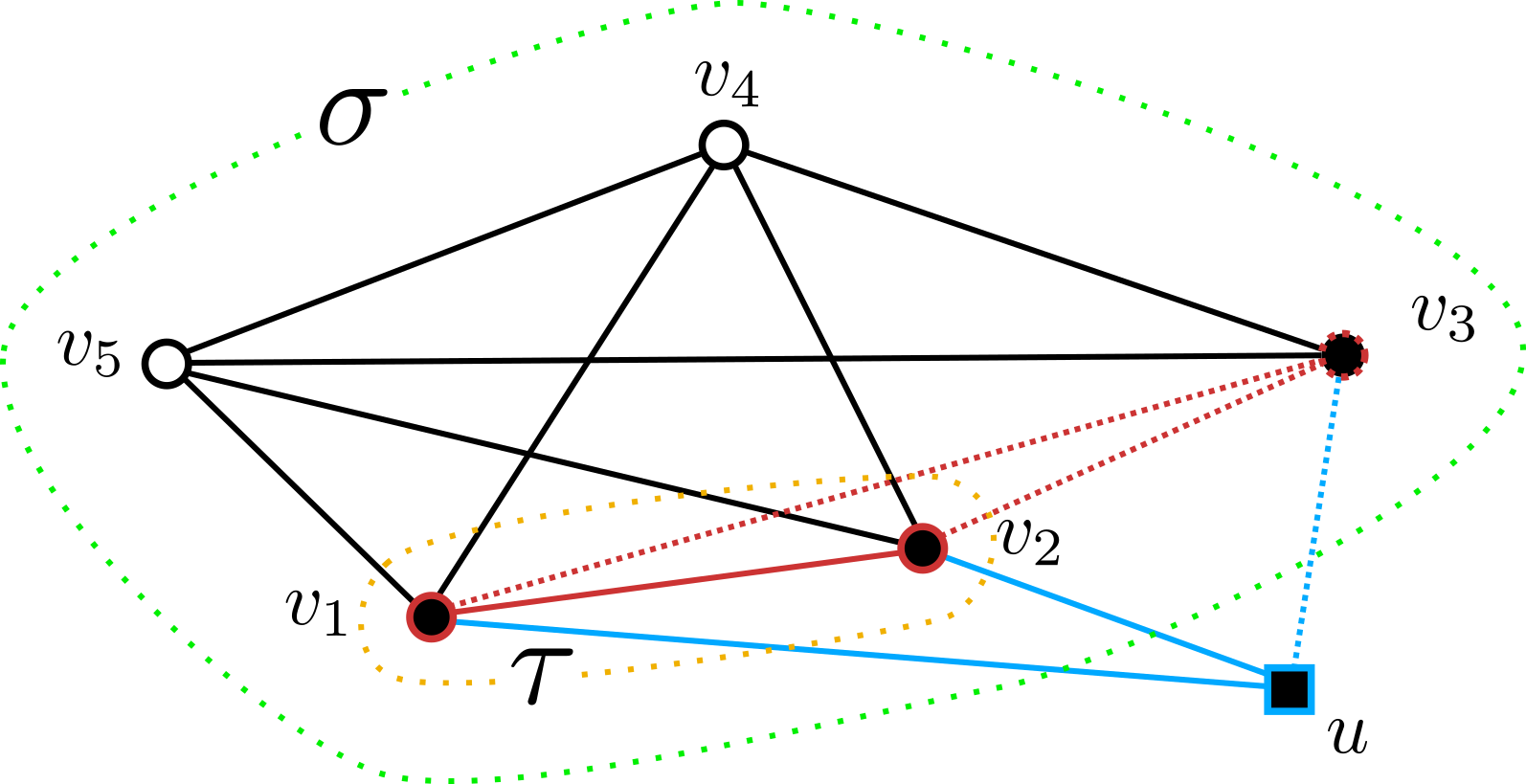}
        \caption{The number of ways to choose $v_{k+1}$ to obtain the figure on the left is $\ell_{k+1}$. Choosing $v_{k+2},\dotsc,v_m \notin \Lk(\{u\})$ such that $(v_1,\dotsc,v_m)$ forms an ordered clique gives the figure on the right, and there are $N_{m,k+1}((v_1,\dotsc,v_{k+1}),u)$ ways to choose these vertices.}
        \label{fig:countingcliques4and5}
    \end{figure}
\end{proof}

Therefore, by Theorem \ref{proposition: supercounting}, we see that our system in (\ref{eq: precountingsystem}) becomes
\[
    \mathcal{G}_m(z) = \ell_0 \dotsb \ell_{m-1} + z\left( \sum_{k = 0}^m \binom{m}{k} N_{m,k}\sum_{\tilde{\tau} \in \widetilde{\mathcal{C}}_{k+1}} \Delta_{\tilde{\tau}}(z)\right)
\]
for $1 \leq m \leq d$. Because
\[
    \sum_{\tilde{\tau} \in \widetilde{\mathcal{C}}_{k+1}} \Delta_{\tilde{\tau}}(z) = \sum_{\tilde{\tau} \in \widetilde{\mathcal{C}}_{k+1}} \Delta_{\tau}(z) = \mathcal{G}_{k+1}(z),
\]
we are left with the system of equations:
\begin{equation}
\label{eq: geodesic system}
    \begin{dcases}
        \mathcal{G}_0(z) \!\!\!\!     & = 1 + z\mathcal{G}_1(z)                                                                                                               \\
        \mathcal{G}_{m}(z) \!\!\!\!   & = \ell_0 \dotsb \ell_{m-1} + z\left( \sum_{k = 0}^m \binom{m}{k} N_{m,k}\mathcal{G}_{k+1}(z)\right) \quad \text{for } 1 \leq m \leq d \\
        \mathcal{G}_{d+1}(z) \!\!\!\! & = 0.
    \end{dcases}
\end{equation}
The clear next step is to bring the recurrence $N_{m,k}$ into a closed form.
\begin{lem}
    \label{lem: closed form lemma}
    For $1 \leq m \leq d$ and $0 \leq k \leq m$, the recurrence $N_{m,k}$ in Prop. \ref{proposition: supercounting}
    has the closed form solution
    \[
        N_{m,k} =
        \begin{dcases}
            \left(\prod_{k < j < m} \ell_j\right)\sum_{j = k}^m \binom{m - k}{j - k} (-1)^{j - k} \ell_j
                        & \text{if $k \neq m$} \\
            \ell_{m - 1} - \ell_m - 1 & \text{if $k = m - 1$} \\
            1 & \text{if $k = m$}.
        \end{dcases}
    \]
\end{lem}
\begin{proof}
    First, we establish the base cases. For any $m \geq 1$, the closed form for $N_{m, m}$ and $N_{m,m-1}$ holds by assumption. Moreover, the recurrence for $N_{m+1, m-1}$ shows that
    \begin{align*}
        N_{m+1,m-1} &= \ell_m N_{m, m-1} - \ell_{m}N_{m+1, m}\\
        &= \ell_m(\ell_{m-1} - \ell_m - 1) - \ell_m(\ell_m - \ell_{m+1} - 1)\\
        &= \ell_m(\ell_{m-1} - 2\ell_m + \ell_{m+1})\\
        &= \left(\prod_{m-1 < j < m+1} \ell_j\right)\sum_{j = m-1}^{m+1} \binom{2}{j - m - 1}(-1)^{j - m - 1}\ell_j.
    \end{align*}
    In particular, we see that the closed form holds for $N_{1,0}$, $N_{1,1}$, $N_{2,0}$, $N_{2,1}$, and $N_{2,2}$.
    
    Now, we proceed by induction on $m$. Fix $m \geq 3$ and suppose that the closed form holds for $N_{m,k}$ for all $0 \leq k \leq m$. We induct backward on $k$ to establish the statement for $N_{m+1,k}$ for all $0 \leq k \leq m + 1$. By the preceding paragraph, we know that for $k \geq m - 1$, the closed form holds. Thus, we fix $k < m - 1$, and suppose that the closed form holds for $N_{m+1,k+1}$. We will show the closed form is then correct for $N_{m+1,k}$. We have
    \begin{align}
        N_{m+1,k} & = \ell_{m} N_{m,k} - \ell_{k+1}N_{m+1,k+1} \notag\\
                  & = \ell_m\left[ \left(\prod_{k < j < m} \ell_j\right)\sum_{j = k}^m \binom{m - k}{j - k} (-1)^{j - k} \ell_j\right] \notag\\
                  & \ \ \ - \ell_{k+1}\left[ \left(\prod_{k+1 < j < m+1} \ell_j\right)\sum_{j = k+1}^{m+1} \binom{m - k}{j - k - 1} (-1)^{j - k - 1} \ell_j \right] \notag\\
                  & = \left(\prod_{k < j < m+1} \ell_j\right)\left[\underbrace{\sum_{j = k}^m \binom{m - k}{j - k} (-1)^{j - k} \ell_j}_{(\ast)} - \underbrace{\sum_{j = k+1}^{m+1} \binom{m - k}{j - k - 1} (-1)^{j - k - 1} \ell_j}_{(\ast\ast)}\right] \label{eq: sum diff induction}
    \end{align}
    The product in (\ref{eq: sum diff induction}) agrees with the closed form solution, so we need only argue that the difference between $(\ast)$ and $(\ast\ast)$ is equal to the sum in the closed form. We rewrite $(\ast)$ as
    \begin{equation}
        \binom{m-k}{0}(-1)^{0}\ell_{k} + \sum_{j = k+1}^{m} \binom{m - k}{j - k} (-1)^{j - k} \ell_j \label{eq: sum ast}
    \end{equation}
    and $(\ast\ast)$ as
    \begin{equation}
        \sum_{j = k+1}^{m} \binom{m - k}{j - k - 1} (-1)^{j - k - 1} \ell_j + \binom{m- k}{m-k}(-1)^{m-k}\ell_{m+1} .\label{eq: sum ast ast}
    \end{equation}
    Taking the difference of (\ref{eq: sum ast}) and (\ref{eq: sum ast ast}) and applying Pascal's rule for binomial coefficients yields
    \[
        \binom{m-k}{0}(-1)^{0}\ell_{k} + \sum_{j = k+ 1}^m \binom{m - k + 1}{j - k}(-1)^{j-k}\ell_j - \binom{m- k}{m-k}(-1)^{m-k}\ell_{m+1}
    \]
    As $\binom{m-k}{0} = \binom{m-k+1}{0}$, $\binom{m-k}{m-k} = \binom{m-k+1}{m-k+1}$, and $-(-1)^{m-k} = (-1)^{m-k+1}$, this is indeed equal to
    \[
        \sum_{j = k}^{m+1} \binom{m - k + 1}{j - k}(-1)^{j-k}\ell_j
    \]
    as desired.

    Thus, the closed form holds for $N_{m+1,k}$ for all $0 \leq k \leq m + 1$. By induction on $m$, the closed form holds for all $N_{m,k}$.
\end{proof}

We will now obtain a solution to the system in (\ref{eq: geodesic system}) in terms of a polynomial recurrence. First, set $a_0 = 1$ and $a_m = \ell_0 \dotsb \ell_{m-1}$ for all $1 \leq m \leq d$. Then, set $b_{0,0} = 1$ and
\[
    b_{m,k} = \binom{m}{k} N_{m,k},
\]
for $1 \leq m \leq d$ and $0 \leq k \leq m$. Observe that this implies $b_{m,m} = 1$. Thus, for $1 \leq m \leq d$, we can rewrite (\ref{eq: geodesic system}) as
\begin{align*}
    \mathcal{G}_m(z) & = a_m + z\sum_{k = 0}^m b_{m,k} \mathcal{G}_{k+1}(z) \notag                                                   \\
                     & = a_m + z\sum_{k=0}^{m-2} b_{m,k} \mathcal{G}_{k+1}(z) + b_{m,m-1}z\mathcal{G}_{m}(z) + z\mathcal{G}_{m+1}(z).
\end{align*}
Hence, by rearranging we get
\begin{equation}
    \label{eq:simplifiedfinalexpression}
    \mathcal{G}_{m+1}(z) = \frac{1}{z}((1-b_{m,m-1}z)\mathcal{G}_m(z) - a_m) -\sum_{k = 0}^{m-2} b_{m,k}\mathcal{G}_{k+1}(z)
\end{equation}
for all $1 \leq m \leq d$. Note that $\mathcal{G}_{d+1}(z) = 0$, but the derivation for (\ref{eq:simplifiedfinalexpression}) still works.

\begin{lem}
    \label{lemma:recurrencepolynomials}
    If $1 \leq m \leq d+1$, then
    \[
        \mathcal{G}_{m}(z) = \frac{1}{z^{m}}(p_{m}(z)\mathcal{G}_0(z) - q_{m}(z)),
    \]
    where the polynomials $p_m(z)$ and $q_m(z)$ are recursively defined by
    \begin{align*}
        p_{m+1}(z) & = p_m(z) - \left(\sum_{k=0}^{m-1} b_{m,k} z^{m-k}p_{k+1}(z)\right)           \\
        q_{m+1}(z) & = q_{m}(z) + z^ma_m - \left(\sum_{k=0}^{m-1} b_{m,k}z^{m-k}q_{k+1}(z)\right)
    \end{align*}
    with $p_1(z) = q_1(z) = 1$.
\end{lem}
\begin{proof}
    For $m = 1$, we have
    \[
        \mathcal{G}_1(z) = \frac{1}{z}(\mathcal{G}_0(z) - 1) = \frac{1}{z}(p_1(z)\mathcal{G}_0(z) - q_1(z))
    \]
    by (\ref{eq: geodesic system}). Now, fix $1 \leq m \leq d$ and suppose that the above formulas for $\mathcal{G}_\ell(z)$, $p_\ell(z)$, and $q_\ell(z)$ hold for all $1 \leq \ell \leq m$. By (\ref{eq:simplifiedfinalexpression}),
    \begin{align*}
        \mathcal{G}_{m+1}(z) & = \frac{1}{z}\left((1-b_{m,m-1}z)\mathcal{G}_m(z) - a_m \right) - \sum_{k = 0}^{m-2} b_{m,k}\mathcal{G}_{k+1}(z)                  \\
                             & = \frac{1}{z}\left((1-b_{m,m-1}z)\frac{1}{z^m}(p_m(z)\mathcal{G}_0(z) - q_m(z)) - a_m\right)                                      \\
                             & \quad - \sum_{k = 0}^{m-2} \frac{b_{m,k}}{z^{k+1}}\left(p_{k+1}(z)\mathcal{G}_0(z) - q_{k+1}(z)\right)                            \\
                             & = \frac{1}{z^{m+1}}\left((1-b_{m,m-1}z)p_m(z)\mathcal{G}_0(z) - (1-b_{m,m-1}z)q_m(z)\right)                                       \\
                             & \ \ \ \ - \frac{1}{z^{m+1}}\sum_{k=0}^{m-2} b_{m,k} z^{m-k}(p_{k+1}(z)\mathcal{G}_0(z) - q_{k+1}(z)) - \frac{1}{z^{m+1}}(z^m a_m)
    \end{align*}
    We combine the fractions above to get a single fraction with denominator $z^{m+1}$. The numerator is given by
    \begin{align*}
         & \left[(1-b_{m,m-1}z)p_m(z) - \left(\sum_{k = 0}^{m-2} b_{m,k}z^{m-k} p_{k+1}(z)\right)\right]\mathcal{G}_0(z)  \\
         & \ \ \ - \left[(1 - b_{m,m-1}z)q_m(z) - \left(\sum_{k=0}^{m-2}b_{m,k} z^{m-k}q_{k+1}(z)\right) + z^m a_m\right] \\
         & = p_{m+1}(z)\mathcal{G}_0(z) - q_{m+1}(z),
    \end{align*}
    as desired.
\end{proof}

From Lemma \ref{lemma:recurrencepolynomials} and the fact that $\mathcal{G}_{d + 1}(z) = 0$, we see that the geodesic growth series of $\RACG(\Gamma)$ is given by
\[
    \mathcal{G}_0(z) = \frac{q_{d+1}(z)}{p_{d+1}(z)}
\]
Finally, we solve the recurrence relations defining the polynomials $p_m(z)$ and $q_m(z)$. The following lemma generalizes the solution for both polynomials.

\begin{lem}
\label{lemma: solving polynomial recurrence}
    Let $x_1, \dots, x_d$ be fixed indeterminates. For $1 \leq m \leq d$, recursively define the polynomials
    \begin{equation}
        \label{eq: recurrence}
        r_{m + 1}(x_1, \dots, x_d, z) = r_m(x_1, \dots, x_d, z) + x_mz^m - \sum_{k = 0}^m b_{m, k} z^{m - k} r_{k + 1}(x_1, \dots, x_d, z)
    \end{equation}
    with the base case $r_1(x_1, \dots, x_d, z) = 1$. To make the following formula simpler, we define $x_0 = 1$. Then $r_m(x_1, \dots, x_d, z)$ is given by
    \[
        r_m(x_1, \dots, x_d, z) = \sum_{i = 0}^{m - 1}
        \sum_{j = 0}^i
        \left(\sum_{\substack{j \leq s_1 < t_1 \leq \dots \leq s_n < t_n \leq m - 1 \\ (t_1 - s_1) + \dots + (t_n - s_n) = i - j}}
        (-1)^n b_{t_1, s_1} \cdots b_{t_n, s_n}\right)
        x_j z^i.
    \]
    (Note that when choosing the numbers $j \leq s_1 < t_1 \leq \dots \leq s_n < t_n \leq m - 1$, we allow for $n = 0$, in which case the corresponding summand $(-1)^n b_{t_1, s_1} \cdots b_{t_n, s_n}$ evaluates to $1$.)
\end{lem}
\begin{proof}
    The above formula gives $r_1(x_1, \dots, x_d, z) = 1$, as required. Inductively, assume for some $1 < m \leq d$ that $r_k(x_1, \dots, x_d, z)$ agrees with the above formula for all $1 \leq k \leq m$. We want to show that $r_{m + 1}(x_1, \dots, x_d, z)$ also agrees with the above formula. From the recurrence formula (\ref{eq: recurrence}) and the inductive hypothesis, we see that the only terms of $r_{m + 1}(x_1, \dots, x_d, z)$ are all scalar multiples of $x_jz^i$ for $j \leq i \leq m$. Hence, it suffices to show that the coefficient of $x_jz^i$ in $r_{m + 1}(x_1, \dots, x_d, z)$ is
    \begin{equation}
        \label{eq: coefficient}
        \sum_{\substack{j \leq s_1 < t_1 \leq \dots \leq s_n < t_n \leq m \\ (t_1 - s_1) + \dots + (t_n - s_n) = i - j}}
        (-1)^n b_{t_1, s_1} \cdots b_{t_n, s_n}
    \end{equation}
    for $j \leq i \leq m$.

    First, consider the case when $j = i = m$. By the inductive hypothesis, the only term in the right-hand side of (\ref{eq: recurrence}) containing $x_m$ is the term $x_mz^m$. Hence, the coefficient of $x_mz^m$ is $1$ in agreement with (\ref{eq: coefficient}).

    Next, consider the case when $j < m$. Observe that a choice of numbers $j \leq s_1 < t_1 \leq \dots \leq s_n < t_n \leq m$ such that $\sum_{\ell = 1}^n (t_\ell - s_\ell) = i - j$ falls into one of two disjoint cases:
    \begin{itemize}
        \item $t_n < m$, in which case we have chosen $j \leq s_1 < t_1 \leq \dots \leq s_n < t_n \leq m - 1$ such that $\sum_{\ell = 1}^n (t_\ell - s_\ell) = i - j$.
        \item $t_n = m$, in which case we have chosen $s_n = k$ for $0 \leq k \leq m - 1$ and $j \leq s_1 < t_1 \leq \dots \leq s_{n - 1} < t_{n - 1} \leq k$ such that $\sum_{\ell = 1}^{n - 1} (t_\ell - s_\ell) = (i - j) - (m - k) = i - m + k - j$.
    \end{itemize}
    Hence, we may break up the sum in (\ref{eq: coefficient}) into the sums
    \[
        \sum_{\substack{j \leq s_1 < t_1 \leq \dots \leq s_n < t_n \leq m \\ (t_1 - s_1) + \dots + (t_n - s_n) = i - j}}
        (-1)^n b_{t_1, s_1} \cdots b_{t_n, s_n}
        = S_1 + \sum_{k = 0}^{m - 1} S_{2, k},
    \]
    where
    \[
        S_1 = \sum_{\substack{j \leq s_1 < t_1 \leq \dots \leq s_n < t_n \leq m - 1 \\ (t_1 - s_1) + \dots + (t_n - s_n) = i - j}}
        (-1)^n b_{t_1, s_1} \cdots b_{t_n, s_n}
    \]
    and
    \[
        S_{2, k} = \sum_{\substack{j \leq s_1 < t_1 \leq \dots \leq s_{n - 1} < t_{n - 1} \leq k \\ (t_1 - s_1) + \dots + (t_{n - 1} - s_{n - 1}) = i - m + k - j}}
        (-1)^n b_{t_1, s_1} \cdots b_{t_{n - 1}, s_{n - 1}}b_{m, k}.
    \]
    Now, by the inductive hypothesis, we have the following facts:
    \begin{itemize}
        \item $S_1$ is the coefficient of $x_jz^i$ in $r_m(x_1, \dots, x_d, z)$.
        \item If $j \leq i - m + k$, $-S_{2, k} / b_{m, k}$ is the coefficient of $x_jz^{i - m + k}$ in $r_{k + 1}(x_1, \dots, x_d, z)$. This implies $S_{2, k}$ is the coefficient of $x_jz^i$ in $-b_{m, k} z^{m - k}r_{k + 1}(x_1, \dots, x_d, z)$.
        \item If $j > i - m + k$, then $r_{k + 1}(x_1, \dots, x_d, z)$ has no $x_jz^{i - m + k}$ term. In this case, $S_{2, k} = 0$, so $S_{2, k}$ is again the coefficient of $x_jz^i$ in $-b_{m, k} z^{m - k}r_{k + 1}(x_1, \dots, x_d, z)$.
    \end{itemize}
    Therefore, the right-hand side of (\ref{eq: recurrence}) agrees with (\ref{eq: coefficient}) (recall that $j < m$, so the $a_mz^m$ term in (\ref{eq: recurrence}) can be ignored). This proves the claim, and so we are done.
\end{proof}

By evaluating the polynomials $r_m(x_1, \dots, x_d, z)$ at $x_1 = \dots = x_d = 0$, we obtain the polynomials $p_m(z)$, and by evaluating at $x_i = a_i$ (recall that $a_i$ are the numbers defined just after the proof of Lemma \ref{lem: closed form lemma}), we obtain the polynomials $q_m(z)$. Therefore, this lemma proves that
\[
    p_m(z) = \sum_{i = 0}^{m - 1} \left(\sum_{\substack{j \leq s_1 < t_1 \leq \dots \leq s_n < t_n \leq m - 1 \\ (t_1 - s_1) + \dots + (t_n - s_n) = i}}
    (-1)^n b_{t_1, s_1} \cdots b_{t_n, s_n}\right)
    z^i
\]
and
\[
    q_m(z) =
    \sum_{i = 0}^{m - 1} \sum_{j = 0}^i
    \left(\sum_{\substack{j \leq s_1 < t_1 \leq \dots \leq s_n < t_n \leq m - 1 \\ (t_1 - s_1) + \dots + (t_n - s_n) = i - j}}
    (-1)^n b_{t_1, s_1} \cdots b_{t_n, s_n}\right)
    a_j z^i.
\]
By substituting the values of $a_m$ and $b_{m, k}$, we obtain the formula given in Theorem \ref{thm: main RACG theorem}. Note that the parenthetical remark at the end of Lemma \ref{lemma: solving polynomial recurrence} justifies the definition $M_{i, j} = 1$ when $i = j$. This completes the proof of Theorem \ref{thm: main RACG theorem}.

\section{Triangle-Free NGPs with constant vertex number}
\label{sec: theorem 3}
Now, we apply the techniques in Section \ref{sec: theorem 2} to compute the geodesic growth series of NGPs. Due to the increased complexity, we restrict ourselves to link-regular numbered graphs $(\Gamma, N)$ such that $\Gamma$ is \textit{triangle-free} (meaning it contains no $3$-cliques), and where $N$ assigns a constant number to every vertex. Since we have already covered the case when $N \equiv 2$, we will further assume $N \geq 3$. Because $N$ is constant, the link-regularity condition on $(\Gamma, N)$ is the equivalent to saying that the simplicial graph $\Gamma$ itself is link-regular (see the argument above Definition \ref{def: link regularity of simplicial graphs}). Moreover, because $\Gamma$ contains no $3$-cliques (and hence no $m$-cliques for all $m \geq 3$), the link-regularity of $\Gamma$ simply states that the \textit{valence} (number of neighbors) of each vertex is constant, i.e., $\Gamma$ is a \textit{regular} graph. We say a regular graph is \textit{$L$-regular} if each vertex has $L$ neighbors.

\begin{thm}
\label{prop: last prop}
Let $\Gamma$ be an $L$-regular and triangle-free simplicial graph with $n$ vertices. Given an integer $N \geq 2$, let $(\Gamma, N)$ be the numbered graph where all vertex numbers are $N$. The geodesic growth series $\mathcal{G}(z)$ of $\NGP(\Gamma)$ can be determined by solving the following system of equations for $1 \leq k, \ell \leq \lfloor N / 2 \rfloor$.
\begin{align*}
    \mathcal{G}(z) &= 2z\mathcal{G}_1(z) + 1 \\
    \mathcal{G}_k(z) &= z(\mathcal{G}_{k + 1}(z)
    + 2\mathcal{G}_{\{1, k\}}(z) + 2(n - L - 1)\mathcal{G}_1(z)) + n \\
    \mathcal{G}_{\{k, \ell\}}(z)
    &= z(\mathcal{G}_{\{k + 1, \ell\}}(z)
    + 2(L - 1)(\mathcal{G}_{\{1, k\}}(z)
    + \mathcal{G}_{\{1, \ell\}}(z))
    + 2 L(n - 2L)\mathcal{G}_1(z))
    + nL.
\end{align*}
Here, $\mathcal{G}_k(z)$ and $\mathcal{G}_{\{k, \ell\}}(z)$ are power series such that $\mathcal{G}_{\lfloor N / 2 \rfloor + 1}(z) = \mathcal{G}_{\{\lfloor N / 2 \rfloor + 1, \ell\}}(z) = 0$ for all $1 \leq \ell \leq \lfloor N / 2\rfloor$.
\end{thm}

\begin{proof}
 Let $\mathcal{D}$ be the FSA given by Theorem \ref{theorem: fsa} that recognizes the geodesic language of $(G, S)$. Applying Proposition \ref{proposition: grammar rules}, we obtain a regular grammar $C$, where for each powered clique $\Sigma$, we have a variable $\mathbf{\Delta}_\Sigma$, and for every transition $\delta(\Sigma, v)$, we have the production rule
\[
    \mathbf{\Delta}_\Sigma \to v\mathbf{\Delta}_{\delta(\Sigma, v)}.
\]
The start variable is $\mathbf{\Delta}_\varnothing$. Arguing as in Lemma \ref{lemma: unambiguous}, one can show that this grammar is unambiguous, and that every variable is reachable.

If $\Sigma$ has support $\{v_1, \dots, v_m\}$, we will write $\mathbf{\Delta}_{\{v_1^{\Sigma(v_1)}, \dots, v_m^{\Sigma(v_m)}\}}$ for $\mathbf{\Delta}_\Sigma$. Note that because $\Gamma$ is triangle-free, the support of $\Sigma$ can contain at most $2$ elements. We may expand the production rules for the start variable as
\begin{align*}
    \mathbf{\Delta}_\varnothing &\to v\mathbf{\Delta}_v && (v \in V\Gamma) \\
    \mathbf{\Delta}_\varnothing &\to v^{-1}\mathbf{\Delta}_{v^{-1 }} && (v \in V\Gamma) \\
    \mathbf{\Delta}_\varnothing &\to \varepsilon.
\end{align*}
Similarly, for any $1 \leq k \leq \lfloor N / 2\rfloor$, we may expand the production rules for $\mathbf{\Delta}_{\{v^k\}}$ as
\begin{align*}
    \mathbf{\Delta}_{\{v^k\}} &\to v\mathbf{\Delta}_{\{v^{k + 1}\}} && (k + 1 \leq \lfloor N / 2\rfloor) \\
    \mathbf{\Delta}_{\{v^k\}} &\to u^{\pm1}\mathbf{\Delta}_{\{u^{\pm1}, v^k\}} && (u \in \Lk(v)) \\
    \mathbf{\Delta}_{\{v^k\}} &\to u^{\pm1}\mathbf{\Delta}_{\{u^{\pm1}\}} && (u \notin \Lk(v), u \neq v) \\
    \mathbf{\Delta}_{\{v^k\}} &\to \varepsilon.
\end{align*}
The production rules for $\mathbf{\Delta}_{\{v^{-k}\}}$ are analogous. Finally, for any $1 \leq k, \ell \leq \lfloor N / 2 \rfloor$, using the fact that $\Gamma$ is triangle-free, we may expand the production rules for $\mathbf{\Delta}_{\{u^k, v^\ell\}}$ as
\begin{align*}
    \mathbf{\Delta}_{\{u^k, v^\ell\}}
    &\to u\mathbf{\Delta}_{\{u^{k + 1}, v^\ell\}} && (k + 1 \leq \lfloor N / 2\rfloor) \\
    \mathbf{\Delta}_{\{u^k, v^\ell\}}
    &\to v\mathbf{\Delta}_{\{u^k, v^{\ell + 1}\}} && (\ell + 1 \leq \lfloor N / 2\rfloor) \\
    \mathbf{\Delta}_{\{u^k, v^\ell\}}
    &\to w^{\pm 1}\mathbf{\Delta}_{\{u^k, w^{\pm 1}\}} && (w \in \Lk(u), w \neq v) \\
    \mathbf{\Delta}_{\{u^k, v^\ell\}}
    &\to w^{\pm 1}\mathbf{\Delta}_{\{v^\ell, w^{\pm 1}\}} && (w \in \Lk(v), w \neq u) \\
    \mathbf{\Delta}_{\{u^k, v^\ell\}}
    &\to w^{\pm 1}\mathbf{\Delta}_{\{w^{\pm 1}\}} && (w \not\in \Lk(u) \cup \Lk(v)) \\
    \mathbf{\Delta}_{\{u^k, v^\ell\}} &\to \varepsilon.
\end{align*}
Finally, the production rules for $\mathbf{\Delta}_{\{u^{-k}, v^\ell\}}$ and  $\mathbf{\Delta}_{\{u^{-k}, v^{-\ell}\}}$ are analogous.

Applying Theorem \ref{theorem: chomsky black magic}, we obtain the system of equations
\begin{align*}
    \Delta_\varnothing(z)
    ={}& z\sum_{v \in V\Gamma}(\Delta_{\{v\}}(z) + \Delta_{\{v^{-1}\}}(z)) + 1 \\
    \Delta_{\{v^k\}}(z)
    ={}& z\left(\underbrace{\Delta_{\{v^{k + 1}\}}(z)}_{\text{ if $k + 1 \leq \lfloor N / 2\rfloor$}}{}
    + \sum_{u \in \Lk(v)} (\Delta_{\{u, v^k\}}(z) + \Delta_{\{u^{-1}, v^k\}}(z))\right. \\
    &\quad \ \ \left.
    + \sum_{u \not\in \Lk(v), u \neq v} (\Delta_{\{u\}}(z) + \Delta_{\{u^{-1}\}}(z))\right)
    + 1 \\
    \Delta_{\{u^k, v^\ell\}}(z)
    ={}& z\left(\underbrace{\Delta_{\{u^{k + 1}, v^\ell\}}(z)}_{\text{ if $k + 1 \leq \lfloor N / 2\rfloor$}}{}
    + \underbrace{\Delta_{\{u^k, v^{\ell + 1}\}}(z)}_{\text{ if $\ell + 1 \leq \lfloor N / 2\rfloor$}}{}\right. \\
    & \quad \ \ + \sum_{w \in \Lk(u), w \neq v} (\Delta_{\{u^k, w\}}(z)
    + \Delta_{\{u^k, w^{- 1}\}}(z)) \\
    & \quad \ \ + \sum_{w \in \Lk(v), w \neq u} (\Delta_{\{v^\ell, w\}}(z) + \Delta_{\{v^\ell, w^{-1}\}}(z)) \\
    & \quad \ \ \left.+ \sum_{w \not\in \Lk(u) \cup \Lk(v)} (\Delta_{\{w\}}(z) + \Delta_{\{w^{-1}\}}(z))
    \right) + 1.
\end{align*}
Here, $\Delta_\Sigma(z)$ denotes the growth series of the language $\mathcal{L}_\Sigma$ consisting of all words that can be derived from $\mathbf{\Delta}_\Sigma$. In particular, $\Delta_\varnothing$ is the geodesic growth series of $(G, S)$. We also have analogous equations for $\Delta_{v^{-k}}(z)$, $\Delta_{\{u^{-k}, v^\ell\}}(z)$ and  $\Delta_{\{u^{-k}, v^{-\ell}\}}(z)$.

We claim that the power series $\Delta_{\{v^k\}}(z)$ and $\Delta_{\{v^{-k}\}}(z)$ are equal. To do this, we must show that $\mathcal{L}_{\{v^k\}}$ and $\mathcal{L}_{\{v^{-k}\}}$ contain the same number of words for each fixed length $n \geq 0$. Let $x \in \mathcal{L}_{\{v^k\}} \cap S^n$. This means the word $v^kx$ is a geodesic word. Let $x'$ be the word obtained from $x$ by replacing every letter $v^{\pm 1}$ with $v^{\mp 1}$. Clearly, $v^{-k}x'$ is still a geodesic word, so $x' \in \mathcal{L}_{\{v^{-k}\}}$. The map $x \mapsto x'$ from $\mathcal{L}_{\{v^k\}} \to \mathcal{L}_{\{v^{-k}\}}$ has inverse $x \mapsto x'$, so it is a bijection, and the claim follows.

Using similar reasoning, we have $\Delta_{\{u^k, v^\ell\}}(z) = \Delta_{\{u^{-k}, v^\ell\}}(z) = \Delta_{\{u^{-k}, v^{-\ell}\}}(z)$. Applying this to the above equations, we obtain the system
\begin{align}
    \label{eq: first}
    \Delta_\varnothing(z)
    ={}& 2z\sum_{v \in V\Gamma}\Delta_{\{v\}}(z) + 1 \\
    \label{eq: second}
    \Delta_{\{v^k\}}(z)
    ={}& z\left(\underbrace{\Delta_{\{v^{k + 1}\}}(z)}_{\text{ if $k + 1 \leq \lfloor N / 2\rfloor$}}{}
    + 2\sum_{u \in \Lk(v)} \Delta_{\{u, v^k\}}(z)
    + 2\sum_{u \not\in \Lk(v), u \neq v} \Delta_{\{u\}}(z)\right)
    + 1 \\
    \label{eq: third}
    \Delta_{\{u^k, v^\ell\}}(z)
    ={}& z\left(\underbrace{\Delta_{\{u^{k + 1}, v^\ell\}}(z)}_{\text{ if $k + 1 \leq \lfloor N / 2\rfloor$}}{}
    + \underbrace{\Delta_{\{u^k, v^{\ell + 1}\}}(z)}_{\text{ if $\ell + 1 \leq \lfloor N / 2\rfloor$}}{}
    + 2\sum_{w \in \Lk(u), w \neq v} \Delta_{\{u^k, w\}}(z)\right. \\
    \notag
    & \quad \ \ + 2\left.\sum_{w \in \Lk(v), w \neq u} \Delta_{\{v^\ell, w\}}(z)
    + 2\sum_{w \not\in \Lk(u) \cup \Lk(v)} \Delta_{\{w\}}(z)
    \right) + 1.
\end{align}

Before proceeding, we introduce some notation. We let $\widetilde{E}\Gamma$ denote the set of \textit{directed edges} of $\Gamma$, i.e., pairs $(u, v) \in (V\Gamma)^2$ such that $\{u, v\} \in E\Gamma$. Let $\mathcal{G}(z)$ denote the geodesic growth series $\Delta_\varnothing(z)$, and let
\[
    \mathcal{G}_k(z) =
    \begin{cases}
    \sum_{v \in V\Gamma} \Delta_{\{v^k\}}(z) & \text{if $1 \leq k \leq \lfloor N / 2\rfloor$} \\
    0 & \text{if $k = \lfloor N / 2\rfloor + 1$,}
    \end{cases}
\]
and
\[
    \mathcal{G}_{\{k, \ell\}}(z) =
    \begin{cases}
        \sum_{(u, v) \in \widetilde{E}\Gamma} \Delta_{\{u^k, v^\ell\}}(z)
        & \text{if $1 \leq k, \ell \leq \lfloor N / 2 \rfloor$,} \\
        0
        & \text{if $k = \lfloor N / 2\rfloor + 1$ or $\ell = \lfloor N / 2\rfloor + 1$.}
    \end{cases}
\]
Finally, it will be convenient to write $\Delta_{(u^k, v^\ell)}(z)$ and $\Delta_{(v^\ell, u^k)}(z)$ for $\Delta_{\{v^k\}}(z)$. Of course, $\Delta_{(u^k, v^\ell)}(z) = \Delta_{(v^\ell, u^k)}(z)$, but this distinction will be useful in the following counting arguments.

Equation (\ref{eq: first}) can now be restated as
\begin{equation}
    \label{eq: first in G}
    \mathcal{G}(z) = 2z\mathcal{G}_1(z) + 1.
\end{equation}
If we sum equation (\ref{eq: second}) over all vertices $v \in \Gamma$, we get
\[
    \mathcal{G}_k(z) = z\left(\mathcal{G}_{k + 1}(z)
    + 2 \sum_{v \in V\Gamma} \sum_{u \in \Lk(v)} \Delta_{(u, v^k)}(z)
    + 2 \sum_{v \in V\Gamma} \sum_{u \not\in \Lk(v), u \neq v} \Delta_{\{u\}}(z)
    \right) + n
\]
for all $1 \leq k \leq \lfloor N / 2\rfloor$. (Note that there is no edge case when $k = \lfloor N / 2 \rfloor$ because $\mathcal{G}_{\lfloor N / 2 \rfloor + 1}(z) = 0$.) By using counting arguments, we can rewrite the two sums in the above equation as follows. The first sum encounters each directed edge $(u, v) \in \widetilde{E}\Gamma$ exactly once, so we may write
\[
    \sum_{v \in V\Gamma} \sum_{u \in \Lk(v)} \Delta_{(u, v^k)}(z)
    = \sum_{(u, v) \in \widetilde{E}\Gamma} \Delta_{(u, v^k)}(z)
    = \mathcal{G}_{1, k}(z).
\]
The second sum ranges over all pairs of non-adjacent, distinct vertices $u, v \in V\Gamma$. For a fixed vertex $u \in V\Gamma$, there are $|V\Gamma \setminus \{u\} \setminus \Lk(u)| = n - L - 1$ other vertices not adjacent to $u$, which means $\Delta_{\{u\}}(z)$ appears in the sum $n - L - 1$ times. Hence, we have
\[
    \sum_{v \in V\Gamma} \sum_{u \not\in \Lk(v), u \neq v} \Delta_{\{u\}}(z)
    = (n - L - 1) \sum_{v \in V\Gamma} \Delta_{\{v\}}(z)
    = (n - L - 1) \mathcal{G}_1(z).
\]
Our equation becomes
\begin{equation}
    \label{eq: second in G}
    \mathcal{G}_k(z) = z(\mathcal{G}_{k + 1}(z)
    + 2\mathcal{G}_{1, k}(z) + 2(n - L - 1)\mathcal{G}_1(z)) + n.
\end{equation}
Summing equation (\ref{eq: third}) over all directed edges, we get
\begin{align*}
    \mathcal{G}_{\{k, \ell\}}(z)
    ={}& z\left(\mathcal{G}_{\{k + 1, \ell\}}(z)
    + \mathcal{G}_{\{k, \ell + 1\}}(z)
    + 2\sum_{(u, v) \in \widetilde{E}\Gamma}\left(\sum_{w \in \Lk(u), w \neq v} \Delta_{(u^k, w)}(z)\right.\right. \\
    & \quad \ \ \left.\left.+ 2\sum_{w \in \Lk(v), w \neq u} \Delta_{(v^\ell, w)}(z)
    + 2\sum_{w \not\in \Lk(u) \cup \Lk(v)} \Delta_{\{w\}}(z)
    \right)\right) + |\widetilde{E}\Gamma|
\end{align*}
for all $1 \leq \ell, k \leq \lfloor N / 2 \rfloor$. (Again, there are no edge cases when $k = \lfloor N / 2\rfloor$ or when $\ell = \lfloor N / 2 \rfloor$ by our definition of $\mathcal{G}_{\{k, \ell\}}(z)$.) First, note that $\widetilde{E}\Gamma$ contains $nL$ elements; indeed, to choose an edge $\{u, v\} \in E\Gamma$, there are $n$ choices for $u \in V\Gamma$, and $L$ choices for $v \in \Lk(u)$. As before, we will rewrite the three double sums in the above equation. For the first sum, fix a directed edge $(u, w) \in \widetilde{E}\Gamma$. The term $\Delta_{(u^k, w)}(z)$ appears once for every directed edge $(u, v) \in \widetilde{E}\Gamma$ such that $w \neq v$. This means $\Delta_{(u^k, w)}(z)$ appears $\abs{\Lk(u) \setminus \{w\}} = L - 1$ times. Thus, we have
\[
    \sum_{(u, v) \in \widetilde{E}\Gamma}\sum_{w \in \Lk(u), w \neq v} \Delta_{(u^k, w)}(z)
    = (L - 1)\sum_{(u, w) \in \widetilde{E}\Gamma} \Delta_{(u^k, w)}(z)
    = (L - 1)\mathcal{G}_{\{1, k\}}(z).
\]
By the same reasoning, the second sum can be rewritten
\[
    \sum_{(u, v) \in \widetilde{E}\Gamma}\sum_{w \in \Lk(v), w \neq u} \Delta_{(v^\ell, w)}(z)
    = (L - 1)\mathcal{G}_{\{1, \ell\}}(z).
\]
For the third sum, fix a vertex $w \in V\Gamma$. The term $\Delta_{\{w\}}(z)$ appears once for every directed edge $(u, v) \in \widetilde{E}\Gamma$ such that $w$ is adjacent to neither $u$ nor $v$. Note that there are $L(n - L - 1)$ choices of $(u, v)$ such that $u$ is not adjacent to $w$: $n - L - 1 = |V\Gamma \setminus \{u\} \setminus \Lk(u)|$ choices for $u$, followed by $L = \abs{\Lk(u)}$ choices for $v$. Of these choices, there are $L(L - 1)$ choices of $(u, v)$ such that $v$ is adjacent to $w$: $L = \abs{\Lk(w)}$ choices for $v$, followed by $L - 1 = \abs{\Lk(v) \setminus \{w\}}$ choices for $u$. Therefore, the term $\Delta_{\{w\}}(z)$ appears $L(n - 2L) = L(n - L - 1) - L(L - 1)$ times; hence, we may write
\[
    \sum_{(u, v) \in \widetilde{E}\Gamma}\sum_{w \not\in \Lk(u) \cup \Lk(v)} \Delta_{\{w\}}(z)
    = L(n - 2L) \sum_{w \in V\Gamma} \Delta_{\{w\}}(z)
    = L(n - 2L)\mathcal{G}_1(z).
\]
Altogether, we get the equation
\begin{equation}
    \label{eq: third in G}
    \mathcal{G}_{\{k, \ell\}}(z)
    = z(\mathcal{G}_{\{k + 1, \ell\}}(z)
    + 2(L - 1)(\mathcal{G}_{\{1, k\}}(z)
    + \mathcal{G}_{\{1, \ell\}}(z))
    + 2 L(n - 2L)\mathcal{G}_1(z))
    + nL.
\end{equation}
The equations (\ref{eq: first in G}), (\ref{eq: second in G}), and (\ref{eq: third in G}) form the system described by the theorem statement, so we are done.
\end{proof}

\section{Further Questions}
\label{sec: further}
Using the techniques detailed in this paper, there is still much that may be found regarding the geodesic growth of NGPs. For instance, in Section \ref{sec: theorem 2}, we solved for the geodesic growth of RACGs, while in Section \ref{sec: theorem 3}, we solved for the geodesic growth of NGPs associated to triangle-free numbered graphs with constant numbering. Combining the techniques in these sections, one may be able to solve for the geodesic growth of NGPs associated to numbered graphs with constant numbering, that may not be triangle-free. There will be more equations to deal with than in our case, since larger clique sizes are allowed. It may furthermore be possible to lift the requirement that the vertex numbers are constant, though the resulting equations will likely be rather unwieldy.

\bibliography{references}
\bibliographystyle{plain}

\end{document}